\documentclass[10pt,twoside,a4paper,leqno]{article}

\usepackage{a4}
\usepackage{amsmath}
\usepackage{amssymb}
\usepackage{amsthm}
\usepackage{amsfonts}
\usepackage{amstext}
\usepackage{amsopn}
\usepackage{amsxtra}
\usepackage{graphicx}
\usepackage[latin1]{inputenc}
\usepackage{mathrsfs}
\usepackage{dsfont}
\usepackage{mathrsfs}
\newcommand\R{{\ensuremath {\mathbb R} }}
\newcommand\C{{\ensuremath {\mathbb C} }}

\newcommand\Z{{\ensuremath {\mathbb Z} }}

\newcommand\1{{\ensuremath {\mathds 1} }}

\renewcommand\phi{\varphi}

\newcommand{\alp}{\boldsymbol{\alpha}}

\newcommand\ii{{\ensuremath {\infty}}}

\def\tr{\mathop{\rm tr}\nolimits} 
\newcommand{\E}{\mathcal{E}}

\newtheorem{thm}{Theorem}[section]
\newtheorem{lemma}[thm]{Lemma}
\newtheorem{corollary}[thm]{Corollary}
\newtheorem{prop}[thm]{Proposition}
\newtheorem{definition}[thm]{Definition}
\newtheorem{remark}[thm]{Remark}

\numberwithin{equation}{section}

\title{\bf A new definition of the Dirac-Fock ground state.}
\author{\bf \'Eric S\'ER\'E$^1$}
\date{\today}
\begin{document}
\maketitle

\begin{center}
\footnotesize $^1$ Université Paris Dauphine, PSL Research University, CNRS, UMR 7534, CEREMADE,\break
Place du Maréchal De Lattre De Tassigny, F-75775 Paris cedex 16, France\break
\texttt{sere@ceremade.dauphine.fr}

\end{center}

\smallskip

\begin{abstract}

The Dirac-Fock (DF) model replaces the Hartree-Fock (HF) approximation in quantum chemistry when relativistic effects cannot be neglected. Since the Dirac operator is not bounded from below, the notion of ground state is problematic in this model, and several definitions have been proposed in the literature. We give a new definition for the ground state of the DF energy, inspired of Lieb's relaxed variational principle for HF. Our definition and existence proof are simpler and more natural than in previous works on DF, but remain more technical than in the nonrelativistic case. One first needs to construct a set of physically admissible density matrices that satisfy a certain nonlinear fixed-point equation: we do this by introducing an iterative procedure, described in an abstract context. Then the ground state is found as a minimizer of the DF energy on this set. 

\end{abstract}

\bigskip

\section{Introduction.}

The Hartree-Fock (HF) model is a mean-field approximation widely used in nonrelativistic quantum chemistry and well understood mathematically (see \cite{LS,Lieb,Lions-87,B,BLLS} and the references in these papers). The Hartree-Fock energy of a system of $q$ electrons near a nucleus of atomic number $Z$ can be defined on the set of projectors of rank $q$ acting in the Hilbert space of one-body electronic states. The HF ground state is defined as a projector $\gamma$ minimizing this energy. It satisfies the self-consistent equation $\gamma=\1_{(-\infty,\mu_{\gamma,q}]}(H_\gamma)$ where $H_\gamma$ is the mean-field Hamiltonian in the presence of the nucleus and of the electrons in the state $\gamma$, $\mu_{\gamma,i}$ being the $i$-th smallest eigenvalue of this Hamiltonian, counted with multiplicity (it was proved in \cite{BLLS} that for the ground state, $\mu_{\gamma,q}<\mu_{\gamma,q+1}$). In \cite{Lieb}, Lieb gave an alternative formulation. He extended the Hartree-Fock functional to the closed convex hull of the set of projectors and proved that for any operator in this hull, there exists a projector having at most the same energy (see also \cite{B} for a simpler proof of Lieb's result). Thanks to this principle, the existence of HF ground states is easily proved by weak lower semicontinuity arguments when $q\leq Z$. This relaxation of constraints also has applications to numerical quantum chemistry. Let us mention, in particular, the ODA algorithm of Canc\`es and Lebris \cite{CanBri-00-2} which has excellent stability properties.\medskip

The Dirac-Fock equations were first introduced by Swirles \cite{Swirles-35}. They are the relativistic analogue of the Hartree-Fock equations
with the positive nonrelativistic Schr\"{o}dinger Hamiltonian $-\Delta/2$ replaced by the free Dirac operator $ D$, a first order operator which is unbounded from below. The corresponding Dirac-Fock energy is also unbounded from below, contrary to the HF energy. This causes serious mathematical and numerical difficulties (see {\it e.g.} \cite{ELS} and references therein). In particular, the Dirac-Fock equations can only be interpreted as stationarity equations of the DF energy.
Despite this issue, they have been widely used in computational atomic physics and quantum chemistry to study heavy elements and their compounds. They allow predictions of atomic and molecular properties in good agreement with experimental data when the correlation effects are not too strong (see {\it e.g.} \cite{Reiher-Wolf} and references therein).\medskip

The free Dirac operator is defined as follows:
\begin{equation}
 D=-i\sum_{k=1}^3\alpha_k\partial_k+\beta :=-i \,\alp\cdot \nabla+\beta
\label{dirac_free}
\end{equation}
where $\alp=(\alpha_1,\alpha_2,\alpha_3)$ and
$$\beta=\left(\begin{matrix}
I_2 & 0\\ 0 & -I_2\\
\end{matrix}\right),\qquad
\alpha_k=\left(\begin{matrix}
0 & \sigma_k\\ \sigma_k & 0\\
\end{matrix}\right),$$
with
$$\sigma_1=\left(\begin{matrix}
0 & 1\\ 1 & 0\\
\end{matrix}\right), \qquad
\sigma_2=\left(\begin{matrix}
  0 & -i\\ i & 0\\
\end{matrix}\right), \qquad
\sigma_3=\left(\begin{matrix}
1 & 0\\ 0 & -1\\
\end{matrix}\right).$$
Here we have taken units such that $\hbar=m=c=1$ where $m$ is the rest mass of the electron.

The operator $ D$, defined on the domain $H^1(\R^3,\C^4)$, is self-adjoint in the Hilbert space ${\cal H}:=L^2(\R^3,\C^4)$. Its form-domain is ${\cal F}:=H^{1/2}(\R^3,\C^4)$, and we can also view $ D$ as a bounded linear operator from ${\cal F}$ to ${\cal F}'=H^{-1/2}(\R^3,\C^4)$. 
The anticommutation relations satisfied by the matrices $\alpha_k$ and $\beta$ ensure that
$$ D^2=-\Delta+1.$$
The spectrum of the self-adjoint operator $ D$ is $\sigma( D)=(-\ii,-1]\cup [1,\ii)$. 
In what follows, the projector associated with the negative (resp. positive) part of 
this spectrum will be denoted by $\Lambda^-$ (resp. $\Lambda^+$):
$$\Lambda^-:=\1_{(-\ii,0)}( D),\qquad \Lambda^+:=\1_{(0,+\ii)}( D).$$
We then have
$$ D\Lambda^-=\Lambda^- D=-\sqrt{1-\Delta}\;\Lambda^-=-\Lambda^-\sqrt{1-\Delta}\,,$$
$$ D\Lambda^+=\Lambda^+ D=\sqrt{1-\Delta}\;\Lambda^+=\Lambda_+\sqrt{1-\Delta}\,.$$

\noindent
We endow the form-domain ${\cal F}$ with the Hilbert-space norm $\Vert \psi\Vert_{\cal F}:=(\psi,\vert  D\vert\psi)^{1/2}\,.$ \medskip

In the whole paper, ${\cal B}(E_1,E_2)$ is the space of bounded linear maps from the Banach space $E_1$ to the Banach space $E_2$; the corresponding norm is $\Vert \cdot \Vert_{{\cal B}(E_1,E_2)}$. We note ${\cal B}(E):={\cal B}(E,E)\,.$ When $E$ is a Hilbert space we also consider the space $\sigma_1(E)$ of trace-class operators on $E$. The associated norm and trace are denoted by $\Vert\cdot\Vert_{\sigma_1(E)}$ and $\tr_E$.\medskip

Let 
\begin{equation}\label{defX}
X:=\{\gamma \in {\cal B}(\mathcal H)\,:\ \gamma =\gamma ^* \  ,  \ (1-\Delta )^{1/4}\gamma (1-\Delta )^{1/4} \in \sigma_1({\cal H})\}\;.
\end{equation}
We endow $X$ with the Banach-space norm

\begin{equation}\label{defnorm}
\Vert \gamma\Vert_X:=\Vert (1-\Delta )^{1/4}\gamma (1-\Delta )^{1/4}\Vert_{\sigma_1(\mathcal H)}\,.
\end{equation}

To each positive integer $q$ we associate the set of projectors
$$\mathcal P_q:=\{\gamma \in X\,: \gamma^2=\gamma\,,\, \tr_{\mathcal H}(\gamma)=q\}\,.$$
The elements of $\mathcal P_q$ are of the form $\gamma=\sum_{k=1}^q \vert\psi_k><\psi_k\vert$ with $\psi_k\in H^{1/2}(\R^3,\C^4)$ and $\langle\psi_k,\psi_l\rangle_{L^2}=\delta_{kl}$. They are the one-body density operators of the $q$-electron Slater determinants $\Psi=\frac{1}{\sqrt{q!}}\psi_1\wedge\cdots\wedge\psi_q$, and we refer to them as \textit{Dirac-Fock projectors}.\medskip

Inspired by Lieb's variational principle \cite{Lieb}, we also associate to any nonnegative real number $q$, the sets
$$\Gamma_{q}:=\{\gamma \in  X \, : \ 0\leq \gamma\leq id_\mathcal H\,\hbox{ and }\ \tr_{\mathcal H}(\gamma)= q\}\;,\ \ \Gamma_{\leq q}:=\bigcup_{0\leq q'\leq q} \Gamma_{q'}\,.$$
We shall refer to the elements of these sets as \textit{Dirac-Fock density operators}.
The set $\Gamma_{\leq q}$ is convex and closed in the weak-$^*$ topology of $X\,$. When $q$ is a positive integer, $\Gamma_{\leq q}$ is the weak-$^*$ closed convex hull of $\mathcal P_q$ in $X$ and the projectors of rank $q$ are its extremal points. Here, the weak-$^*$ topology of $X\,$ is the smallest topology such that for any compact operator $Q:\,\mathcal H\to\mathcal H$, the linear form $\ell_Q:\,\gamma\in X\mapsto \tr_{\mathcal H}(Q(1-\Delta )^{1/4}\gamma(1-\Delta )^{1/4})$ is continuous.\medskip

The electrons are exposed to an external Coulomb field $V=-\alpha\, \mathfrak n * {1\over \vert x \vert}$ generated by a nuclear charge distribution $\mathfrak n$. We assume that $\mathfrak n$ is a positive and finite Radon measure on ${\R}^3$. Its total mass $Z:=\int_{{\R}^3}d\mathfrak n\,$ represents the number of protons in the molecule.  In our system of units, $\alpha=\frac{e^2}{4\pi\varepsilon_0\hbar c}$ is a dimensionless constant. Its physical value is approximately $1/137$. The energy of a Dirac-Fock density operator $\gamma$ is
$${\cal E}_{DF}(\gamma):={\rm tr}\bigl( ( D+V) \gamma \bigr)+{\alpha\over 2}\int_{{\R}^3\times{\R}^3} {\rho_\gamma (x)\rho_\gamma (y)-{\rm tr}_{{\C}^4}(\gamma(x,y)\gamma(y,x))\over \vert x-y\vert}dx\,dy\,.$$
The quadratic term in this energy comes from the repulsive electrostatic interaction between electrons. It depends on the integral kernel $\gamma(x,y)$ of the trace-class operator $\gamma$ and on its charge density $\rho_\gamma(x):={\rm tr}_{{\C}^4}\gamma(x,x)$.
Due to the presence of the Dirac operator $ D$, ${\cal E}_{DF}$ is {\it not} bounded from below on $\Gamma_q$, contrary to the nonrelativistic HF energy.
The functional ${\cal E}_{DF}$ is well-defined and smooth on $X$. Its differential at $\gamma$ is the linear form $h\in X\mapsto {\rm tr}( D_{V,\gamma} h)$, with
$$ D_{V,\gamma}:= D+V +\alpha W_\gamma$$
and
$$W_\gamma \psi (x):=\bigl(\rho_\gamma * {1\over \vert x \vert}\bigr) \, \psi(x)-\int_{{\R}^3}{\gamma(x,y)\psi(y)\over \vert x-y\vert} dy\,.$$
 If $\Vert V  D^{-1}\Vert_{{\cal B}({\cal H})}<1$, the operator $ D_{V,\gamma}$ is self-adjoint in ${\cal H}$, with same domain, form-domain and essential spectrum as $ D$. Note that by Hardy's inequality, a sufficient condition for the inequality $\Vert V  D^{-1}\Vert_{{\cal B}({\cal H})}<1$ is $2\alpha Z<1$. For larger values of $\alpha Z$ this inequality does not necessarily hold, but $ D_{V,\gamma}$ is still self-adjoint with domain $H^1(\R^3,\C^4)$ if $\alpha Z<\frac{\sqrt{3}}{2}$, while for $\frac{\sqrt{3}}{2}\leq Z<1$, this operator has a distinguished self-adjoint realization in  ${\cal H}$, whose domain is a subspace of $H^{1/2}(\R^3,\C^4)$ (see e.g. \cite{Thaller,ELS2} and references therein).\medskip

Note that in general, for $\gamma$ in $X$, $( D+V) \gamma$ does not make sense as a trace-class operator in $\mathcal H$, so the expression ${\rm tr}\bigl( ( D+V) \gamma \bigr)$ should be interpreted as ${\rm tr}_{\mathcal H}\bigl( \vert D\vert^{1/2} \gamma\vert D\vert^{1/2}{\rm sign}( D) \bigr)+\alpha\int_{\mathbb{R}^3}V\rho_\gamma\,.$ A similar interpretation should be made for ${\rm tr}( D_{V,\gamma} h)$. Such an abuse of notation is common in the mathematical literature on Hartree-Fock theory (see {\it e.g.} \cite[Remark 2.2]{Solovej-03}) and we make it throughout the paper. \medskip

We now introduce the Dirac-Fock equation, as a stationarity condition on $\mathcal E_{DF}$ under unitary transformations of $\mathcal H$. If $A$ is a bounded self-adjoint operator on $\mathcal H$, we may define the unitary flow $U(t)={\rm exp}(-itA)$. If, in addition, the operator $(1-\Delta)^{-1/4}A(1-\Delta)^{1/4}$ is bounded on $\mathcal H$ then, for each $\gamma\in\Gamma_{q}$, $U(t)\gamma U(-t)$ is in $\Gamma_{q}$ and we may define the function $f_A(t):=\mathcal E_{DF}(U(t)\gamma U(-t))$. The derivative of this function at $t=0$ is $f'_A(0)=i\tr( D_{V,\gamma}[\gamma,A])=i\tr([ D_{V,\gamma},\gamma]A)$. So, one has $f'_A(0)=0$ for all $A$ if and only if $\gamma$ is a solution of the {\it Dirac-Fock equation}
$$[ D_{V,\gamma},\gamma]=0\,.$$

From the physics viewpoint, the operator $ D_{V,\gamma}$ represents the Hamiltonian of a relativistic electron in the mean field generated by the nuclei and the one-body operator $\gamma$. The spectrum of $ D_{V,\gamma}$ contains the infinite interval of negative energies $(-\infty,-1]$. To deal with this difficulty, one may introduce the spectral projectors\medskip

$$P^{\pm}_{V,\gamma}:=\1_{{\R}_\pm}( D_{V,\gamma})\;.$$
With this notation, $P^\pm_{V,0}=\1_{{\R}_\pm}( D+V)\,$ and $P^\pm_{0,0}=\Lambda^\pm\;.$\medskip

The negative spectral subspace $P^-_{V,\gamma} \mathcal H$ is the Dirac sea in the presence of the nuclei and electrons. According to Dirac's interpretation of negative energy states, physical electrons should be orthogonal to their own Dirac sea.
This leads us to define, for $q\in\Z_+$, the set of {\it admissible} Dirac-Fock projectors
$$\mathcal P_q^+:=\{\gamma \in  \mathcal P_q \ : \  P^+_{V,\gamma} \gamma = \gamma\}\,,$$
and, for $q\in\R_+$, the sets of admissible Dirac-Fock density operators
$$\Gamma^+_{q}:=\{\gamma \in  \Gamma_{q} \ : \  P^+_{V,\gamma} \gamma = \gamma\}\;,\ \ \Gamma^+_{\leq q}:=\bigcup_{0\leq q'\leq q} \Gamma^+_{q'}\,.$$
The elements of these sets can be interpreted as the one-body density operators of particle conserving quasi-free states (see \cite{BLS}), the underlying one-particle Hilbert space being $P^+_{V,\gamma} \mathcal H$.\medskip

To take into account the orthogonality to the Dirac sea, we must write the Dirac-Fock equation in the more restrictive form
$$ [ D_{V,\gamma}\,,\,\gamma]=0\;,\;\gamma\in\Gamma_q^+\,.$$

In relativistic quantum chemistry, one is particularly interested in Dirac-Fock ground states. By analogy with the nonrelativistic theory, it is tempting to define such states (for $q\in\Z_+$) as the solutions of the self-consistent equation
$$ \gamma=\1_{(0,\mu]}( D_{V,\gamma})\;\hbox{ with }\mu\hbox{ such that  }{\rm tr}_{\mathcal H}(\gamma)=q \,.$$
Such a fixed-point equation naturally leads to an iterative algorithm, well-known in computational quantum chemistry under the name of {\it Roothaan self-consistent field (SCF) method}. However, even in the nonrelativistic case, the SCF scheme does not always converge and when it does, there is no guarantee that one has found a ``true" ground state, that is, a minimizer of the Hartree-Fock energy (see \cite{CanBri-00-1}). The situation is worse with the DF functional, since $\mathcal E_{DF}$ is not bounded from below on $\Gamma_q$.\medskip

Note that in the physical and chemical literature, the DF functional is usually defined on the set $\mathcal P_q$ of Dirac-Fock projectors (for $q\in\Z_+^*$) and is written as a function of an orthonormal sequence of monoelectronic states $\Psi=(\psi_1,\cdots,\psi_q)$ that generates the range of the Dirac-Fock projector $\gamma$. This point of view was adopted in the mathematical works \cite{ES1,P} where solutions of the Dirac-Fock equations were found as min-max critical points of the energy $\mathcal E_{DF}(\Psi)$. The property $\gamma\in\mathcal P_q^+$ was not imposed as an {\it a priori} constraint, it was an {\it a posteriori} consequence of the min-max method in which the constraints $\langle\psi_k,\psi_l\rangle_{L^2}=\delta_{kl}$ were replaced by a penalization. There was no direct way of defining a ground state in this framework, since there was no minimization principle at hand, except in the weakly relativistic regime \cite{ES2,ES3} that is, when $\alpha$ is very small. Note that in \cite{ES2,ES3}, the conditions on $\alpha$ were not made explicit. This would have been possible in principle, but the result would certainly have been very far from the physical value $1/137$. An alternative approach was introduced by Huber and Siedentop in \cite{HS} and provided the existence of a ground state in the regime of weak interaction between electrons thanks to a fixed-point procedure, for an explicit range of (small) values of $\alpha$. The physical value $1/137$ was not in this range, but not by far in the case of highly ionized atoms. Another work where a simple definition of the ground state is given and its existence proved, is the paper \cite{CN} by Coti Zelati and Nolasco where a one-electron atom with self-generated electromagnetic field is considered. A concavity argument allows these authors to define a reduced energy functional that is bounded from below. However it does not seem easy to extend their elegant construction to multi-electronic problems.\medskip

A physical derivation of the DF model as a mean-field approximation of QED was proposed by Mittleman \cite{Mittleman-81}. This derivation leads to a max-min definition of the ground state. One first considers an infinite-rank projector, and one minimizes the Dirac-Fock energy on a corresponding set of projected states. Then, in a second step, one maximizes the resulting minimum by varying the projector. Unfortunately, such a procedure does not always give solutions of the DF equations: a rigorous justification of the first step (minimization among projected states) has been given in \cite{BFHS}, but negative results on the second step (maximization among projectors) for $q>1$ can be found in \cite{BES,BHS}. Another approach was initiated by Chaix and Iracane \cite{CI}, who derived from QED the Bogoliubov-Dirac-Fock mean-field approximation that takes into account the polarization of the Dirac sea, neglected by Mittleman. Note, however, that in the BDF energy of \cite{CI} an important one-body term was missing. This was corrected in \cite{HLSo} by Hainzl, Lewin and Solovej who gave a more rigorous derivation thanks to a thermodynamic limit procedure. From the point of view of mathematics, the main advantage of BDF over DF is that the energy is bounded from below when defined in a suitable functional framework (see \cite{CIL,BBHS,HLS1,HLS2,HLSo,HLSS}), so the definition of a ground state becomes straightforward and general existence results can be obtained for positive ions and neutral molecules \cite{HLS3},\cite{GraLS} thanks to Lieb's variational principle. But the BDF ground state is not trace-class, an ultraviolet regularization is necessary in order to define its energy and a charge renormalization is needed to correctly interpret the Euler-Lagrange equation.
\bigskip

 Our new definition of a DF ground state avoids the delicate min-max procedure of \cite{ES1,P} as well as the complicated functional framework of BDF, and the associated existence result has a domain of validity much larger than in \cite{ES2,ES3,HS,CN}, that includes the physical value of $\alpha$ and certain multi-electronic atoms.
  
 \begin{definition} To any positive real number $q$ we associate the energy
 $$E_q:=\inf_{\gamma\in\Gamma^+_{\leq q}} \big( {\cal E}_{DF}(\gamma)-{\rm tr}_{\mathcal H}(\gamma) \,\big) \,.$$
 If an admissible Dirac-Fock density operator $\gamma_*\in\Gamma^+_{q}$ is such that $\mathcal E_{DF}(\gamma_*)- q=E_q$,
we call it ``Dirac-Fock ground state of particle number $q$ in the external field $V$".
\end{definition}

The main result of this paper is the existence of a Dirac-Fock ground state of particle number $q$ for positive ions and neutral molecules, under a smallness assumption on $V$ and $\alpha q\,$:
 
 \begin{thm}[Existence of a ground state]\label{main}

Let us introduce the constants
\begin{equation*}
\kappa:=\Vert V  D^{-1}\Vert_{{\cal B}({\cal H})}+2\alpha\,q\,\ \hbox{ and }\ \lambda_0:=1-\alpha\max(q,Z)\,.
\end{equation*}

Assume that $Z,\,q$ and $\,1-\kappa -{\pi\over 4}\alpha\,q\,$ are positive, and that the following condition is satisfied:

\begin{equation}\label{condition}
\alpha\, Z\,<\frac{2}{\pi/2+2/\pi}\ \ \ \hbox{ and }\ \ \ \pi\alpha q \,< \,2(1-\kappa)^{1\over 2}\lambda_0^{1\over 2}\Big( 1-\kappa-{\pi\over 4}\alpha\,q\Big)^{1\over 2}\,.
\end{equation}

 Then:\medskip
 
 $\bullet$  $E_q$ is negative and attained, that is, there exists an admissible Dirac-Fock density operator $\gamma_*\in \Gamma_{\leq q}^+$ such that
$$ {\cal E}_{DF}(\gamma_*)-{\rm tr}_{\mathcal H}(\gamma_*)=E_q \,<\, 0 \,.$$

$\bullet$ For any such minimizer, there is an energy level $\mu\in (0,1]$ such that
\begin{equation}\label{Euler-Lagrange}
\gamma_*=\1_{(0,\mu)}( D_{V,\gamma_*})+\delta\,\ \hbox{ with }
\ 0\leq \delta \leq \1_{\{\mu\}}( D_{V,\gamma_*})\,.
\end{equation}

$\bullet$ If $q< Z$ then $\mu<1\,.$\medskip

$\bullet$ If $q\leq Z$ then $\,\tr_{\mathcal H}(\gamma_*)=q\,$, so $\gamma_*$ is a Dirac-Fock ground state of particle number $q$ in the external field $V$, moreover the following strict binding inequalities hold:

\begin{equation}\label{cc}
\forall \, q'\in (0,q)\;, \ \  E_q<E_{q'}\,.
\end{equation}

\end{thm}

\begin{remark} Our definition of the ground state energy involves $\mathcal E_{DF}-\tr_{\mathcal H}$ instead of $\mathcal E_{DF}$. Physically, this corresponds to subtracting the rest mass of the electron from the mean-field Hamiltonian $ D_{V,\gamma}$: the eigenvalues of the resulting operator are negative, as in the nonrelativistic case.
This subtraction plays a very important role in the proof of Theorem \ref{main}. Without it, the infimum $E_q$ would be attained at $\gamma=0$. One could of course think of subtracting $\lambda\tr_{\mathcal H}$ for some $\lambda<1$ instead of $\lambda=1$, but then one would not be able to guarantee that $\tr_{\mathcal H}(\gamma_*)= q$ when $q\leq Z$. \end{remark}

\begin{remark} Hardy's inequality immediately implies that $\kappa\leq 2\alpha (Z+q)$: see \eqref{hardyV}. Using this estimate and taking $\alpha\approx {1\over 137}$ we find that the smallness assumption \eqref{condition} is satisfied by neutral atoms up to $Z = 22$. For positive ions the situation is better: when $q=2$ in particular, our assumptions are satisfied for $2\leq Z\leq 63$. To deal with heavier elements, one could for instance try to replace Hardy's inequality with refined estimates on the Dirac-Coulomb operator such as those obtained in the papers \cite{BruRohSie},\cite{MorMul-17}.  We leave this question for future research.\end{remark}

\begin{remark} In the case $q>Z$, it follows from our proof that if \eqref{condition} and \eqref{cc} hold true, then $\,\tr_{\mathcal H}(\gamma_*)=q\,$ (see Proposition \ref{binding and ground state}). However, when $q>Z$ we are not able to check \eqref{cc} even for $q-Z$ very small.
\end{remark}

 \begin{remark}\label{nonconvex}

The scalar $\mu-1$ is the Lagrange multiplier associated with the constraint $\tr_{\mathcal H}(\gamma)\leq q\,.$ Contrary to the HF situation \cite{Lieb,BLLS}, for $q\in\Z_+$ we are not able to prove that the highest occupied energy level $\mu$ of the mean-field operator $ D_{V,\gamma_*}$ is full and that the one-body density matrix $\gamma_*$ is a Dirac-Fock projector. The main difficulty is that the spectral projector $P^+_{V,\gamma}$ depends on $\gamma$ in a complicated way and the set $\,\Gamma^+_{\leq q}$ on which we minimize does not seem to be convex.
\end{remark}

In order to prove that the minimizer $\gamma_*$ exists and satisfies the Euler-Lagrange equation \eqref{Euler-Lagrange}, we are first going to construct a $C^1$ retraction $\theta$ of a certain subset $\overline{\mathcal V}$ of $ \Gamma_{\leq q}$ onto $\overline{\mathcal V}\cap \Gamma^+_{\leq q}$. The word {\it retraction} means that $\theta(\overline{\mathcal V})=\overline{\mathcal V}\cap \Gamma^+_{\leq q}$ and $\theta(\gamma)=\gamma$ for all $\gamma$ in $\overline{\mathcal V}\cap \Gamma^+_{\leq q}$. The construction of $\theta$ involves an iterative procedure: for $\gamma \in \overline{\mathcal V}$, taking $\gamma_0=\gamma$ and $\gamma_{p+1}=P^+_{V,\gamma_p}\gamma_p P^+_{V,\gamma_p}$,  $\theta (\gamma)$ is the limit of the sequence $(\gamma_p)$ for the topology of $X$. As we will see, the condition \eqref{condition} guarantees that the set $\overline{\mathcal V}$ is large enough to contain the sublevel set $\{\gamma\in\Gamma^+_{\leq q}\,:\;\mathcal E_{DF}(\gamma) -\tr_{\mathcal H}(\gamma)\leq 0\}$. This is an important point in the proof, since we will also see that $E_q$ is negative.
\medskip

The paper is organized as follows. In Section \ref{retraction}, the existence and regularity properties of $\theta$ are studied  by first constructing this retraction in an abstract context under general assumptions, then checking these assumptions in the case of the Dirac-Fock problem. In Section \ref{existence}, Theorem \ref{main} and Proposition \ref{atoms} are proved thanks to the construction of the preceding section.\medskip

An unpublished version of the present paper is mentioned in the work \cite{FLT}, where our new definition of the ground state is used to study the Scott correction in atoms. In \cite{CMPS}, the existence of solutions to the Dirac-Fock equations in crystals is proved by combining the method of the present work with new compactness arguments. In the recent work \cite{Meng}, the relationship between the Dirac-Fock model and Mittelman's approach is studied, thanks to refined estimates on our retraction $\theta$ and the associated ground state energy in the regime $\alpha<\!<1$.

\section{The retraction $\theta.$}\label{retraction}

We recall that a retraction of the metric space $(F,{\rm d})$ onto one of its subsets $A$ is a continuous map $\theta:\,F\to A$ such that $\theta(x)=x\,,\;\forall x\in A$.

\subsection{An abstract construction.}

We start with an abstract construction valid in any complete metric space.
\begin{prop}\label{abstract-cont-theta}

Let $(F,{\rm d})$ be a complete metric space and $T:\,F \to \,F$ a continuous map. We assume that
$$\exists k\in (0,1)\; , \;\forall x\in F \;, \quad {\rm d}(T^2(x),T(x))\, \leq\, k\,{\rm d}(T(x),x)\; .$$

Then for any $x\in F$ , the sequence $(T^p(x))_{p\geq 0}$ has a limit $\theta(x)\in {\rm Fix}(T)\,$ with the estimate
\begin{equation}\label{conv}
{\rm d}(\theta(x),T^p(x))\,\leq\,{k^p\over 1-k}\,{\rm d}(T(x),x)\,.
\end{equation}

The continuous map $\theta$ obtained in this way is a retraction of $F$ onto ${\rm Fix}(T)\,$, i.e., for any $\,x\in F\,:$ $\,T\circ\theta(x)=\theta(x)$ and for any $\,y\in {\rm Fix}(T)\,:$ $\,\theta(y)=y$.\end{prop}

\begin{proof}
This proposition is a generalisation of Banach's fixed point theorem. For the convergence of $T^n(x)$ to a fixed point, the proof is very similar: by induction one shows that ${\rm d}(T^{p+1}(x),T^p(x))\, \leq\, k^p\,{\rm d}(T(x),x)\,$, so that $(T^p(x))$ is a Cauchy sequence, with the estimate
\begin{equation}\label{Cauchy}
{\rm d}(T^{p+q}(x),T^p(x))\,\leq\,{k^p\over 1-k}\,{\rm d}(T(x),x)\,.
\end{equation}
By completeness of $F$ we conclude that $T^n(x)$ has a limit that we denote $\theta(x)$. By continuity of $T$, $\theta(x)\in {\rm Fix}(T)$. Passing to the limit $q\to\infty$ in \eqref{Cauchy}, we obtain the desired estimate \eqref{conv}.
Moreover, if $x\in {\rm Fix}(T)$ then the sequence $T^n(x)$ is constant, so $\theta(x)=x$.\medskip
 
Now, for any $a\in F$, by continuity of $T$ there is a radius $r(a)>0$ such that 
$$\sup_{x\in B(a,r(a))}{\rm d}(T(x),x) <\infty\,.$$
Then \eqref{conv} implies that the sequence of continuous functions $(T^n)$ converges uniformly to $\theta$ on $B(a,r(a))$, hence the continuity of $\theta$ on $F=\cup_{a\in F} B(a,r(a))$.
 \end{proof}

Note that $T$ is not necessarily a contraction, so in general ${\rm Fix}(T)$ is {\it not} reduced to a point and $\theta$ need not be constant, contrary to what happens with Banach's fixed point theorem. For instance, if $F=X$ is a Hilbert space and $T$ the projection on a closed convex subset $\mathcal C$ of $X$ then for any $x$ , $T^2(x)=T(x)$. The assumptions of Proposition \ref{abstract-cont-theta} are thus trivially satisfied and we just have $\theta=T$, ${\rm Fix}(T)=\mathcal C$.\medskip

We now want to study the differentiability of $\theta$ in a suitable framework.
We consider a Banach space $X$ and we take an open subset ${\cal U}$ of $X$. We assume that $T$ is defined on the closure $F$ of $\mathcal U$. If $Y$ is a Banach space (possibly equal to $X$), we say that a differentiable function $\Phi:\,{\cal U}\to Y$ is in $C^{1, \rm{unif}}({\cal U},Y)$ if its differential $d\Phi$ is uniformly continuous from ${\cal U}$ to ${\cal B}(X, Y)$. We also say that $\Phi\in C^{1,\rm{lip}}({\cal U},Y)$ if $d\Phi$ is Lipschitzian on ${\cal U}$. We have the following regularity result:

\begin{prop}\label{abstract-regular}

Let ${\cal U}$ be a nonempty open subset of a Banach space $X\,$ and let $F$ be the closure of $\,{\cal U}$ in $X$.
Let $T\in C^{0}(F\,,\,{X})\cap C^{1,\rm{lip}}({\cal U}\,,\,{X})$ be such that $T({\cal U})\subset {\cal U}\,$,
$\sup_{x\in {\cal U}} \Vert T(x)-x\Vert_X<\infty\,$, $\sup_{x\in {\cal U}} \Vert dT(x)\Vert_{\mathcal{B}(X)}<\infty\,$
and
$$\exists k\in (0,1)\; , \;\forall x\in {\cal U}\;, \quad \Vert T^2(x)-T(x)\Vert_{X} \leq\, k\Vert T(x)-x\Vert_{X}\; .$$

Then for each $x\in{\cal U}\,,$ the sequence $(d(T^p)(x))_{p\geq 0}$ has a limit $\ell(x)\in {\cal B}({X})$ for the norm $\Vert\cdot\Vert_{{\cal B}({X})}\,,$ this convergence being uniform in $x$. As a consequence, the function $\theta:\,F \to \rm{Fix(T)}\subset F$ constructed thanks to Proposition \ref{abstract-cont-theta} is in $C^{1, \rm{unif}}({\cal U}\,,\,{X})$ and we have $d\theta(x)=\ell(x)$ for all $x\in{\cal U}$.

\end{prop}
 
 The end of this section is devoted to the proof of Proposition \ref{abstract-regular}. {\bf In the sequel, we use the same notation $M$ for several finite constants which only depend on ${\cal U}$ and $T$.}\medskip
 
 We first study the behaviour of $d(T^p)(x)$ for $x$ in ${\rm Fix}(T)\cap {\cal U}$ and $p$ a nonnegative integer. In this case, $d(T^p)(x)$ coincides with the $p$-th power of $dT(x)\,.$

\begin{lemma}\label{fixed-points}

Under the assumptions of Proposition \ref{abstract-regular}, we have an estimate of the form
$$\forall p,q\in\Z_+\,,\,\forall x\in {\rm Fix}(T)\cap {\cal U}\,,\;\Vert dT(x)^{p+q}-dT(x)^p\Vert_{{\cal B}(X)} \leq\, M\, k^p .$$
So, for any $x\in {\rm Fix}(T)\cap {\cal U}\,,$ the sequence $(dT(x)^p)_{p\geq 0}$ has a limit $\ell(x)$ in ${\cal B}(X)$ and the convergence is uniform in $x$:
$$\Vert dT(x)^p\Vert_{{\cal B}(X)}\leq\, M\;\hbox{ and }\;\Vert \ell(x)-dT(x)^p\Vert_{{\cal B}(X)}\leq\, M\, k^p.$$

\end{lemma}

\begin{proof}

Given $x\in {\rm Fix}(T)\cap {\cal U}$ and $h\in X$, for $t\in {\R}$ nonzero and small enough, 
$$\Big\Vert {T^{2}(x+th)-T(x+th)\over t}\Big\Vert_X \,\leq\, k\,\Big\Vert {T(x+th)-x-th\over t}\Big\Vert_X\;.$$
Since $x=T(x)=T^2(x)$ we infer
$$\Big\Vert {T^{2}(x+th)-T^2(x)\over t}-{T(x+th)- T(x) \over t}\Big\Vert_X \,\leq\, k\,\Big\Vert {T(x+th)-T(x)\over t}-h\Big\Vert_X$$
and passing to the limit as $t$ goes to zero:
$$\Vert dT(x)^2 h-dT(x)h\Vert_X \,\leq\, k\,\Vert dT(x)h-h\Vert_X\,.$$
Taking $h=DT(x)^{p}\tilde{h}$, this inequality becomes
$$\Vert (dT(x)^{p+2}-dT(x)^{p+1}) \tilde{h}\Vert_X \,\leq\, k\,\Vert (dT(x)^{p+1}-dT(x)^{p}) \tilde{h}\Vert_X\,,$$
hence
$$\Vert dT(x)^{p+1}-dT(x)^{p}\Vert_{\mathcal B(X)} \,\leq\, k^p\,\Vert dT(x)-id_X\Vert_{\mathcal B(X)}\,.$$
Using the triangle inequality, one infers that
$$\Vert dT(x)^{p+q}-dT(x)^{p}\Vert_{\mathcal B(X)} \,\leq\, \frac{k^p}{1-k}\,\Vert dT(x)-id_X\Vert_{\mathcal B(X)}\,,$$
hence the lemma, since $x\mapsto \Vert dT(x)\Vert_{{\cal B}(X)}$ is bounded on ${\cal U}$ .
\end{proof}

We now consider an arbitrary point $x$ in ${\cal U}$.

\begin{lemma}\label{theta-bounded}

Under the assumptions of Proposition \ref{abstract-regular}, we have an estimate of the form
$$\forall p\in\Z_+\,,\;\forall x \in {\cal U}\,,\; \Vert d(T^p)(x) \Vert_{{\cal B}(X)}\leq M\;.$$

\end{lemma}

\begin{proof}

We denote by $L$ the Lipschitz constant of $dT$ on ${\cal U}$:
$$\forall x,y \in {\cal U}\, , \;\,\Vert dT(x)-dT(y)\Vert_{{\cal B}(X)} \leq\, L\Vert x-y\Vert_X\,.$$
Take $x\in{\cal U}$. With $\delta _i:=dT(T^{i-1}(x))-dT(\theta(x))\,$, we get

\begin{equation}
\Vert\delta_i\Vert_{{\cal B}(X)}\leq L\Vert T^{i-1}(x)-\theta(x)\Vert_X\leq M\, k^i.
\label{delta}
\end{equation}
From Lemma \ref{fixed-points} we also have an estimate of the form

\begin{equation}
\Vert dT(\theta(x))^q\Vert_{{\cal B}(X)}\leq\, M\,.
\label{boundDTtheta}
\end{equation}

\noindent
Now,
\begin{equation*}
\begin{array}{l}
d(T^p)(x)\\
=   (dT(\theta(x))+\delta_p)\circ\cdots\circ(dT(\theta(x))+\delta_1)\\
=  \sum\limits_{j\in[\![0,p]\!] \atop p\geq i_1>\cdots>i_j>i_{j+1}=0}dT(\theta(x))^{p-i_1}\circ\prod\limits_{\mu=1}^j\left(\delta_{i_\mu}\circ  dT(\theta(x))^{i_{\mu}-i_{\mu+1}-1}\right)
\end{array}
\end{equation*}
hence, using the estimates (\ref{delta}) and (\ref{boundDTtheta}) :
\begin{eqnarray*}
\Vert d(T^p)(x)\Vert_{{\cal B}(X)} &\leq&     \sum\limits_{j=0}^p \;M^{2j+1} \sum\limits_{p\geq i_1>\cdots>i_j\geq 1}k^{i_1+\cdots+i_j}\\
&\leq& M\sum_{j=0}^p M^{2j}{(\sum_{i=1}^p k^i)^j\over j!}\,\leq \,M\, {\rm exp}\Big({M^2 k\over 1-k}\Big)\end{eqnarray*} 
and the lemma follows.
\end{proof}

To end the proof of Proposition \ref{abstract-regular}, we show that $(d(T^p)(x))_{p\geq 0}$ is a Cauchy sequence, uniformly in $x\in{\cal U}\,.$

\begin{lemma}\label{cauchy}

Under the assumptions of Proposition \ref{abstract-regular}, the following estimate holds:

$$\forall x\in {\cal U}\;,\;\forall p,q\geq 0\,,\quad \Vert d (T^{p+q})(x)-d (T^p)(x)\Vert_{{\cal B}(X)} \leq M k^{p/2}\;.$$

\noindent
So $d(T^p)(x)$ converges to some $\ell(x)\in {\cal B}(X)$ and $\Vert \ell(x)-d (T^p)(x)\Vert_{{\cal B}(X)} \leq M k^{p/2}\,.$

\end{lemma}

\begin{proof}

Let $m,\,n,\,q$ be nonnegative integers. As in the proof of Lemma \ref{theta-bounded}, we consider $\delta _i:=dT(T^{i-1}(x))-dT(\theta(x))\,.$ For $x\in{\cal U}$ we may write

$$d(T^{m+n+q})(x)-d(T^{m+n})(x)=(A_{m,n+q}(x)+B_{n,q}(x)-A_{m,n}(x))\circ d(T^m)(x)$$
with
\begin{eqnarray*}
A_{m,r}(x)&:=&d(T^{r})(T^m(x))-dT(\theta(x))^{r}\\
&\,=&(dT(\theta(x))+\delta_{m+r})\circ\cdots\circ(dT(\theta(x))+\delta_{m+1})-dT(\theta(x))^r \\
&\,=&\sum\limits_{j\in[\![1,r]\!] \atop r\geq i_1>\cdots>i_j>i_{j+1}=0}dT(\theta(x))^{r-i_1}\circ\prod\limits_{\mu=1}^j\left(\delta_{m+i_\mu}\circ  dT(\theta(x))^{i_{\mu}-i_{\mu+1}-1}\right)
 \end{eqnarray*}
\noindent
and $\,B_{n,q}(x):= dT(\theta(x))^{n+q}-dT(\theta(x))^n\,.$\medskip

\noindent
Using the estimates (\ref{delta}) and (\ref{boundDTtheta}) as in the proof of Lemma \ref{theta-bounded}, we find

\begin{eqnarray*}
\Vert A_{m,r}(x)\Vert_{{\cal B}(X)} &\leq&  \sum_{j=1}^r \;M^{2j+1} \sum\limits_{r\geq i_1>\cdots>i_j\geq 1}k^{mj+i_1+\cdots+i_j}\\
&\leq& M\sum_{j=1}^r (M^2 k^m)^{j}{(\sum_{i=1}^r k^i)^j\over j!}\leq \, M \,\Big [ {\rm exp}\Big({M^2 k^{m+1}\over 1-k}\Big)-1\Big ] \end{eqnarray*}

\noindent which gives an estimate of the form $\Vert A_{m,r}(x)\Vert_{{\cal B}(X)} \leq\,M\,k^m$ for another constant $M\,.$ On the other hand, from Lemma \ref{fixed-points}, $\,\Vert B_{n,q}(x)\Vert_{{\cal B}(X)}\,\leq \,M\, k^n\,.$ From Lemma \ref{theta-bounded}, $\,\Vert d(T^m)(x) \Vert_{{\cal B}(X)}\leq\, M\;.$ Combining these estimates, we find
$$\Vert d(T^{m+n+q})(x)-d(T^{m+n})(x)\Vert_{{\cal B}(X)}\,\leq \,M\,(k^n+k^m)\,.$$
Taking $p=n+m$ with $n=m$ or $n=m+1\,,$ we get the desired estimate
$$\Vert d(T^{p+q})(x)-d(T^{p})(x)\Vert_{{\cal B}(X)}\leq \,M\,k^{p/2}\,.$$
\noindent
This ends the proofs of Lemma \ref{cauchy} and Proposition \ref{abstract-regular}.

\end{proof}

\subsection{Application to Dirac-Fock.}\label{Appli}

From now on, we work in the Banach space $(X,\,\Vert \cdot \Vert_X)$ given by formulas (\ref{defX}) and (\ref{defnorm}) of the introduction. We recall our notations $P^\pm_{V,\gamma}=\1_{\R_\pm}( D_{V,\gamma})$, $\kappa=\Vert V  D^{-1}\Vert_{{\cal B}({\cal H})}+2\alpha\,q\,$ and $\lambda_0=1-\alpha\max(q,Z)$. Our map $T$ will be given by the formula
\begin{equation}\label{defT}
T(\gamma):=P^+_{V,\gamma}\gamma P^+_{V,\gamma}\,.
\end{equation} 
We will see that if $\kappa<1$ then the map $T$ is well-defined from $\Gamma_{\leq q}$ to itself. But to discuss the differentiability of $T$, it is convenient to extend this function to an open neighborhood of $\Gamma_{\leq q}$. So we take a small number $r>0$ (to be chosen later) and we define the open set
$$\Gamma_{\leq q}^r:=\{\gamma\in X\,:\,\hbox{dist}_{\sigma_1(\mathcal H)}(\gamma,\Gamma_{\leq q})<r\}\,.$$

The goal of this subsection is to build an open subset ${\cal U}$ of $\Gamma_{\leq q}^r$ invariant under $T$, satisfying the assumptions of Proposition \ref{abstract-regular} and containing all the admissible Dirac-Fock density operators $\gamma\in \Gamma^+_{\leq q}$ such that $\mathcal E_{DF}(\gamma)\leq \tr_{\mathcal H}(\gamma)$. This will be done under some conditions on $\alpha,\,q,\,V$ and for $r$ small enough.\medskip

We start with a lemma gathering estimates that will be used in the sequel:

\begin{lemma}\label{hardy}

Let $\gamma\in X\,.$\medskip

$\bullet$ The following Hardy-type estimates hold:

\begin{align}
\label{hardy1}\max\Big(\Big\|\rho_{\gamma} * \frac{1}{|\cdot |}\Big\|_{_{\infty}},\,  \left\|W_{\gamma}\right\|_{_{\mathcal B(\mathcal H)}},\, \Big\|\frac{\gamma(x, y)}{|x-y|}\Big\|_{_{\mathcal B(\mathcal H)}}\big)&\leq {\pi\over 2}\big\Vert \,(-\Delta )^{\frac{1}{4}}\gamma \, (-\Delta )^{\frac{1}{4}}\big\Vert_{_{\sigma_1({\cal H})}}\,,\\
\label{hardy2}\Vert W_\gamma (-\Delta )^{-\frac{1}{2}}\Vert_{{\cal B}({\cal H})}&\leq 2\Vert \gamma \Vert_{\sigma_1({\cal H})}\,,\\
\label{hardyV}\Vert V (-\Delta )^{-\frac{1}{2}}\Vert_{{\cal B}({\cal H})}&\leq 2\alpha Z\,.\end{align}
\medskip

$\bullet$ If $\kappa_r:=\kappa+2\alpha r$ is smaller than $1$ and $\,\Vert \gamma \Vert_{\sigma_1({\cal H})}\leq q+r\,\,$ then:
\begin{align}
\label{hardy2bis}\big\Vert \,\vert D_{V,\gamma}\vert^s\,\vert D\vert^{-s}\big\Vert_{{\cal B}({\cal H})}&\leq(1+\kappa_r)^s\;, \quad \forall\,0< s\leq 1\,,\\
\label{hardy3} \big\Vert \,\vert D\vert^{s}\,\vert D_{V,\gamma}\vert^{-s}\big\Vert_{{\cal B}({\cal H})}&\leq \,(1-\kappa_r)^{-s}\;, \quad \forall\,0< s\leq 1\,,\\
\label{commutator}\big\Vert\, \vert  D \vert^{-\frac{1}{2}}P^+_{V,\gamma} \vert  D \vert^{\frac{1}{2}}\big\Vert_{{\cal B}({\cal H})}&\leq\,\left(\frac{1+\kappa_r}{1-\kappa_r}\right)^{\frac{1}{2}}\;.\end{align}
\medskip

$\bullet$ If $\,\alpha\, {\rm max}(q+r,Z+r)<\frac{2}{\pi/2+2/\pi}\,$ and $\,\gamma\in\Gamma_{\leq q}^r\,\,$ then, with the notation $\lambda_r:=\lambda_0-\alpha r$,
\begin{align}\label{hardy4}
{\rm inf}\vert\sigma ( D_{V,\gamma})\vert \geq \lambda_ r >0\;.
\end{align}

\end{lemma}
\bigskip

\begin{proof}

If $\gamma\in X$ then $(-\Delta )^{\frac{1}{4}}\gamma \, (-\Delta )^{\frac{1}{4}}\in \sigma_1({\cal H})$ and $\Vert (-\Delta )^{\frac{1}{4}}\gamma \, (-\Delta )^{\frac{1}{4}}\Vert_{\sigma_1({\cal H})}
\leq \Vert\gamma\Vert_{_X}$, since $\Vert(-\Delta )^{\frac{1}{4}}\,(1-\Delta )^{-1/4}\Vert_{{\cal B}({\cal H})}\leq 1$.\medskip

We may thus write $(-\Delta )^{\frac{1}{4}}\gamma (-\Delta )^{\frac{1}{4}}=\sum_{n=0}^{\infty} d_n\vert\varphi_n><\varphi_n\vert$ where $(\varphi_n)$ is orthonormal in ${\cal H}$, $d_n\in\R$ and $\sum_{n=0}^{\infty} \vert d_n\vert=\Vert (-\Delta )^{\frac{1}{4}}\gamma \, (-\Delta )^{\frac{1}{4}}\Vert_{\sigma_1({\cal H})}\,.$\medskip

For each $n$, we define $\tilde{\varphi}_n=(-\Delta )^{-\frac{1}{4}}\varphi_n$. Then $W_\gamma=\sum_{n=0}^{\infty} \gamma_n W_{\vert\tilde{\varphi}_n><\tilde{\varphi}_n\vert}\,.$ For each $n$, the operator of multiplication by $|\tilde{\varphi}_n |^2 * \frac{1}{|\cdot |}$, the exchange operator of kernel $\frac{\tilde{\varphi}_n(x)\otimes \tilde{\varphi}_n^*(y)}{|x-y|}$ and their difference $W_{\vert\tilde{\varphi}_n><\tilde{\varphi}_n\vert}$ are symmetric and positive on ${\cal H}$, so, by the Cauchy-Schwarz inequality, in order to prove \eqref{hardy1} we just need to show that for any $\psi\in{\cal H}$, $\big\langle\psi,(|\tilde{\varphi}_n |^2 * \frac{1}{|\cdot |})\psi\big\rangle_{{\cal H}} \leq \frac{\pi}{2}\Vert \psi\Vert^2_{\cal H}\,.$ This is done thanks to the Kato-Herbst inequality $\int_{\R^3}\frac{\vert f\vert^2}{\vert x\vert}\leq \frac{\pi}{2}\int_{\R^3} \big\vert\, (-\Delta)^{\frac{1}{4}} f\big\vert^2$ \cite{Kato1}:
$$\Big\langle\psi,\big(|\tilde{\varphi}_n |^2 * \frac{1}{|\cdot |}\big)\psi\Big\rangle_{{\cal H}}= \int \frac{\vert\psi\vert^{2}(x)\vert\tilde{\varphi}_n\vert^{2}(y)}{\vert x-y\vert}dxdy\leq \frac{\pi}{2} \int\vert\psi\vert^{2}(x) \Vert \varphi_n\Vert_{\cal H}^2dx=\frac{\pi}{2}\Vert \psi\Vert^2_{\cal H}\,.$$

Now, in order to prove \eqref{hardy2} we write $\gamma=\sum_{n=0}^{\infty} \gamma_n\vert f_n\rangle\langle f_n\vert$ where $(f_n)$ is orthonormal in ${\cal H}$, $\gamma_n\in\R$ and $\sum_{n=0}^{\infty} \vert \gamma_n\vert=\Vert\gamma\Vert_{\sigma_1({\cal H})}\,.$
Taking $\psi$ in $\dot{H}^1(\R^3,\C^4)$ and $\chi$ in ${\cal H}$, we have
$$\vert \langle\chi,W_\gamma \psi\rangle_{_{\cal H}}\vert \leq \sum_{n=0}^{\infty} \vert \gamma_n\vert\, \vert\langle\chi,W_{\vert f_n\rangle\langle f _n\vert}\,\psi\rangle_{_{\cal H}}\vert \,.$$
Denoting by $\psi^\alpha$ ($1\leq \alpha \leq 4$) the components of a four-spinor $\psi$ and by $\overline{z}$ the conjugate of a complex number $z$, we have
\begin{align*}
&\ \ \vert\langle\chi,W_{\vert f_n\rangle\langle f_n\vert}\,\psi\rangle_{_{\cal H}}\vert\\
&=\,\frac{1}{2}\left\vert \iint\sum_{\alpha,\beta} \frac{\det\left(\begin{matrix}
\bar{f}_n^\alpha(x) & \bar{\chi}^\alpha(x)\\ \bar{f}_n^\beta(y) & \bar{\chi}^\beta(y)\\
\end{matrix}\right)\det\left(\begin{matrix}
f_n^\alpha(x) & \psi^\alpha(x)\\ f_n^\beta(y) & \psi^\beta(y)\\
\end{matrix}\right)}{\vert x-y\vert}dxdy\right\vert\\
&\leq \frac{1}{2} \left(\iint\sum_{\alpha,\beta} \left\vert \det\left(\begin{matrix}
f_n^\alpha(x) & \chi^\alpha(x)\\ f_n^\beta(y) & \chi^\beta(y)\\
\end{matrix}\right)\right\vert^2 dxdy\right)^{\frac{1}{2}}\left(\iint\sum_{\alpha,\beta} \frac{\left\vert \det\left(\begin{matrix}
f_n^\alpha(x) & \psi^\alpha(x)\\ f_n^\beta(y) & \psi^\beta(y)\\
\end{matrix}\right)\right\vert^2}{\vert x-y\vert^2}dxdy\right)^{\frac{1}{2}}\\
&=\left(\iint \vert f_n(x)\vert^2\vert \chi(y)\vert^2-\vert \langle f_n(x),\chi(y)\rangle\vert^2\right)^{\frac{1}{2}}\left(\iint \frac{\vert f_n(x)\vert^2\vert \psi(y)\vert^2-\vert \langle f_n(x),\psi(y)\rangle\vert^2}{\vert x-y\vert^2}\right)^{\frac{1}{2}}\\
&\leq 2\Vert \chi \Vert_{_{\cal H}}\,\big\Vert \,(-\Delta)^{1/2} \psi \big\Vert_{{\cal H}}
\end{align*}
by Hardy's inequality. Estimate \eqref{hardy2} follows.\medskip

To prove estimate \eqref{hardyV} one just needs to write
$$\Vert V\psi\Vert_{{\cal H}}=\alpha\left\Vert \int_{\R^3} \frac{\psi}{\vert\cdot-y\vert}d\mathfrak n(y)\right\Vert_{{\cal H}}\leq \alpha\int_{\R^3} \left\Vert\frac{\psi}{\vert\cdot-y\vert}\right\Vert_{{\cal H}} d\mathfrak n(y)\,\leq\, 2\alpha\,Z\,\big\Vert \,(-\Delta)^{1/2}  \psi \big\Vert_{{\cal H}}\,.$$

Now, by the triangle inequality and \eqref{hardy2}, we have
\begin{equation}\label{hardy2bis stronger}
\Vert  D_{V,\gamma}\psi\Vert_{{\cal H}}\leq \bigl(1+\Vert V  D^{-1}\Vert_{{\cal B}({\cal H})}+2\alpha\,\Vert\gamma\Vert_{\sigma_1({\cal H})}\bigr)\Vert  D\psi\Vert_{{\cal H}}
\,.
\end{equation}
If $\Vert\gamma\Vert_{\sigma_1({\cal H})}\leq q+r\,$, recalling that $\kappa_r=\Vert V  D^{-1}\Vert_{{\cal B}({\cal H})} +2\alpha (q+r)\,,$ we may thus write $\,
 D_{V,\gamma}^2 \leq \bigl(1+\kappa_r\bigr)^2  D^2\,,
$
hence, by interpolation, $\vert  D_{V,\gamma} \vert^{2s}\leq\bigl(1+\kappa_r\bigr)^{2s}\vert  D \vert^{2s}$ for all $0<s\leq 1$: this estimate is the same as \eqref{hardy2bis}. Assuming that $\kappa_r<1$, one proves \eqref{hardy3} in a similar way.
Since $P^+_{V,\gamma}$ commutes with $\vert  D_{V,\gamma}\vert^{1/2}$, estimate \eqref{commutator} directly follows from \eqref{hardy2bis}, \eqref{hardy3} for $s=1/2$.\medskip

To prove \eqref{hardy4} we remark that for each $\gamma$ in $\Gamma_{\leq q}^r$ one has $\tr_{\mathcal H} (\gamma^+)< q+r$, $\tr_{\mathcal H} (\gamma^-)<r$ with $\gamma^\pm=\pm\gamma\1_{\R_\pm}(\gamma)$. Then, using Tix' inequality \cite{Tix-97}\cite{Tix-98} as in Lemma 3.1 of \cite{ES1}, we find that if ${\rm max}(q+r,Z+r)<\frac{2}{\pi/2+2/\pi},$ $\gamma\in\Gamma_{\leq q}^r$ and $\psi\in \hbox{Dom}( D_{V,\gamma})\setminus\{0\}$ then
\begin{align*}
\Vert\psi\Vert_{{\cal H}}\Vert D_{V,\gamma}\psi\Vert_{{\cal H}}&\geq 
\langle\Lambda^+\psi-\Lambda^-\psi,\, D_{V,\gamma}\Lambda^+\psi+ D_{V,\gamma}\Lambda^-\psi\rangle_{{\cal H}}\\
&=\langle\Lambda^+\psi,\, D_{V,\gamma}\Lambda^+\psi\rangle_{{\cal H}}-\langle\Lambda^-\psi,\, D_{V,\gamma}\Lambda^-\psi\rangle_{{\cal H}}\\
&\geq\langle\Lambda^+\psi,\, D_{V,-\gamma^-}\Lambda^+\psi\rangle_{{\cal H}}-\langle\Lambda^-\psi,\, D_{V,\gamma^+}\Lambda^-\psi\rangle_{{\cal H}}\\
&> (1-\alpha\max(q+r,Z+r))\Vert\psi\Vert_{{\cal H}}^2\,.
\end{align*}

The lemma is thus proved.
\end{proof}

We now study the dependence of $P^+_{V,\gamma}$ on $\gamma$.

\begin{lemma}\label{regularity}

With the notations $\kappa_r,\,\lambda_r$ of Lemma \ref{hardy}, assume that $\kappa_r < 1\,$ and $\alpha\, (Z+r)<\frac{2}{\pi/2+2/\pi}\,$, and let
$$a_r := {\pi\alpha\over 4}(1-\kappa_r)^{-1/2}\lambda_r^{-1/2}.$$

Then the map
$$Q:\,\gamma   \mapsto (P^+_{V,\gamma} -P^+_{V,0})$$
is in $C^{1,\rm{lip}}(\Gamma_{\leq q}^r,Y)$ with $Y:={\cal B}({\cal H},{\cal F})\,$ (recalling that ${\cal F}=H^{1/2}(\R^3,\C^4)$ is the form-domain of $ D$) and we have the estimates
\begin{align}
\label{Q lipschitz}&\forall\gamma\,,\,\gamma'\in \Gamma_{\leq q}^r \,\;:\; \Vert Q (\gamma')-Q(\gamma)\Vert_Y \leq a_r \Vert \gamma'-\gamma\Vert_X \\
\label{Q C1 lipschitz}&\forall \gamma\,,\,\gamma' \in \Gamma_{\leq q}^r \,\;:\; \Vert dQ(\gamma')-dQ(\gamma)\Vert_{{\cal B}(X,Y)} \leq K\alpha^2 \Vert \gamma'-\gamma \Vert_X
\end{align}
where $K$ is a positive constant which remains bounded when $\kappa_r$ stays away from $1$.

\end{lemma}

\begin{proof}
The proof consists in calculations similar to those of \cite[Lemma 1]{GLS} or \cite[Lemma 1]{HS}. One writes, for $\gamma,\gamma'\in \Gamma_{\leq q}^r$ and $\chi,\,\psi\in {\cal H}$:
\begin{align*}
&\langle\chi, \vert  D\vert^{1/2}(Q(\gamma')-Q(\gamma))\psi\rangle_{\cal H}\\
&=\frac{\alpha}{2\pi} \int_{\R} \langle\chi, \vert  D\vert^{1/2}( D_{V,\gamma}+iz)^{-1}W_{\gamma'-\gamma}( D_{V,\gamma'}+iz)^{-1}\psi\rangle_{\cal H}dz
\end{align*}
hence, with $\check{\chi}:= \vert  D_{V,\gamma}\vert^{-1/2} \vert  D\vert^{1/2}\chi\,$,
\begin{align*}
&\vert\langle\chi, \vert  D\vert^{1/2}(Q(\gamma')-Q(\gamma))\psi\rangle_{\cal H}\vert\\
&\leq\frac{\alpha\Vert W_{\gamma'-\gamma}\Vert_{{\cal B}({\cal H})}}{2\pi} \left(\int_{\R} \langle\check{\chi}, \vert  D_{V,\gamma}\vert ( D_{V,\gamma}^2+z^2)^{-1}\check{\chi}\rangle_{\cal H}dz\right)^{\frac{1}{2}} \left(\int_{\R} \langle\psi, ( D_{V,\gamma'}^2+z^2)^{-1}\psi\rangle_{\cal H}dz\right)^{\frac{1}{2}}\\
&=\frac{\alpha\Vert W_{\gamma'-\gamma}\Vert_{{\cal B}({\cal H})}}{2}\Vert \check{\chi}\Vert_{_{\cal H}} \, \Vert \,\vert  D_{V,\gamma'}\vert^{-1/2} \psi\Vert_{_{\cal H}}\,.
\end{align*}
From \eqref{hardy3} we have
$$\Vert \check{\chi}\Vert_{{\cal H}}\leq (1-\kappa_r)^{-1/2}\Vert \chi\Vert_{{\cal H}}$$
and from \eqref{hardy4} we have
$$\Vert \,\vert  D_{V,\gamma'}\vert^{-1/2} \psi\Vert_{{\cal H}}\leq \lambda_r^{-1/2}\Vert \psi\Vert_{{\cal H}}\,.$$
As a consequence
$$ \Vert Q (\gamma')-Q(\gamma)\Vert_Y \leq \frac{\alpha}{2}(1-\kappa_r)^{-1/2} \lambda_r^{-1/2}\Vert W_{\gamma'-\gamma}\Vert_{{\cal B}({\cal H})}\,,$$
and from \eqref{hardy1} we have 
$$\Vert W_{\gamma'-\gamma}\Vert_{{\cal B}({\cal H})}\leq \frac{\pi}{2} \Vert \gamma'-\gamma\Vert_X\,,$$
so the estimate \eqref{Q lipschitz} of Lemma \ref{regularity} is proved. Moreover, taking $\gamma'=0$ in \eqref{Q lipschitz} we find that $Q(\gamma)\in Y$, since $Q(0)=0\,.$ Thus, up to now we have proved that $Q$ is a Lipschitz map from $\Gamma_{\leq q}^r$ to $Y$ with Lipschitz constant $a_r$.\medskip

Noting $\dot{\gamma}:=\gamma'-\gamma$ and pushing the expansion of $Q(\gamma')-Q(\gamma)$ one step further, one gets
\[
Q(\gamma')-Q(\gamma)=\mathcal{L}_\gamma(\dot{\gamma}) + R_\gamma(\dot{\gamma})
\]
where
\[ \mathcal{L}_\gamma(\dot{\gamma}):=\frac{\alpha}{2\pi} \int_{\R} ( D_{V,\gamma}+iz)^{-1}W_{\dot{\gamma}}( D_{V,\gamma}+iz)^{-1}dz\,, \]
\[R_\gamma(\dot{\gamma}):=-\frac{\alpha^2}{2\pi} \int_{\R} ( D_{V,\gamma}+iz)^{-1}W_{\dot{\gamma}}( D_{V,\gamma}+iz)^{-1}W_{\dot{\gamma}}( D_{V,\gamma+\dot{\gamma}}+iz)^{-1}dz\,.
\]
Then, using estimates similar to the ones above, one finds that $\mathcal{L}_\gamma$ is in ${\cal B}(X,Y)$ with
$\Vert \mathcal{L}_\gamma \Vert_{{\cal B}(X,Y)}\leq a_r$ and
\begin{align*}
\Vert R_\gamma(\dot{\gamma})\Vert_{Y}&\leq \frac{\alpha^2}{2}(1-\kappa_r)^{-1/2} \lambda_r^{-1/2}\sup_{z\in\R}\Vert W_{\dot{\gamma}}( D_{V,\gamma}+iz)^{-1}W_{\dot{\gamma}}\Vert_{{\cal B}({\cal H})}\\
&\leq \frac{\pi^2\alpha^2}{8}(1-\kappa_r)^{-1/2} \lambda_r^{-3/2}\Vert \dot{\gamma} \Vert^2_{X}\,.
\end{align*}
As a consequence, $Q$ is differentiable at $\gamma$ and $dQ(\gamma)=\mathcal{L}_\gamma\,.$\medskip

Finally, for $h\in X$ one writes
\[(\mathcal{L}_{\gamma'}-\mathcal{L}_\gamma)h=A_\gamma(\dot{\gamma},h)+B_\gamma(\dot{\gamma},h)\]
with

\[A_\gamma (\dot{\gamma},h):=-\frac{\alpha^2}{2\pi} \int_{\R} ( D_{V,\gamma+\dot{\gamma}}+iz)^{-1}W_{h}( D_{V,\gamma+\dot{\gamma}}+iz)^{-1}W_{\dot{\gamma}}( D_{V,\gamma}+iz)^{-1}dz
\]
and
\[B_\gamma(\dot{\gamma},h):=-\frac{\alpha^2}{2\pi} \int_{\R} ( D_{V,\gamma+\dot{\gamma}}+iz)^{-1}W_{\dot{\gamma}}( D_{V,\gamma}+iz)^{-1}W_{h}( D_{V,\gamma}+iz)^{-1}dz\,.
\]
Proceeding as before with each of these expressions, one gets
\[ \Vert  (\mathcal{L}_{\gamma'}-\mathcal{L}_\gamma)h\Vert_Y  \leq \frac{\pi^2\alpha^2}{4}(1-\kappa_r)^{-1/2} \lambda_r^{-3/2}\Vert \dot{\gamma} \Vert_{X}\Vert h \Vert_{X}\,.\]
The estimate \eqref{Q C1 lipschitz} follows, with $K:=\frac{\pi^2}{4}(1-\kappa_r)^{-1/2} \lambda_r^{-3/2}$. So $Q\in C^{1,\rm{lip}}(\Gamma_{\leq q}^r,Y)$ and the lemma is proved.
\end{proof}

We are now able to study the map $T$. Our first result is:

\begin{prop}\label{contraction-principle}

Assume that $\kappa_r < 1\,,$ $\alpha\, (Z+r)<\frac{2}{\pi/2+2/\pi}\,$ and let $a_r$ be as in Lemma \ref{regularity}.\medskip

Then the map $T:\, \gamma \to P^+_{V,\gamma} \gamma P^+_{V,\gamma}$
is well-defined from $\Gamma_{\leq q}$ to itself and from $\Gamma_{\leq q}^r$ to itself, and
for any $\gamma\in \Gamma_{\leq q}^r\,:$\bigskip

\begin{equation}\label{main-estimate} 
\begin{aligned}
&\Vert T^2(\gamma)-T(\gamma)\Vert_X \\
&\leq\, 2a_r\left(\,\Vert T(\gamma)\,\vert  D\vert^{1/2}\Vert_{\sigma_1({\cal H})}+\frac{a_r (q+r)}{2}\Vert T(\gamma)-\gamma\Vert_X\right)\Vert T(\gamma)-\gamma\Vert_X\,.
\end{aligned}
\end{equation}

Moreover $T$ is differentiable on $\Gamma_{\leq q}^r\subset X$ and there are two positive constants $C_{\kappa,r}\,,\, L_{\kappa,r}$ such that, for all $ \gamma,\,\gamma' \in \Gamma_{\leq q}^r \,$:

\begin{eqnarray}
&\Vert dT(\gamma) \Vert_{{\cal B}(X)}\leq C_{\kappa,r}\Big(1+\alpha\,\big\Vert \gamma\,\vert  D\vert^{1/2}\big\Vert_{\sigma_1(\mathcal{H})}\Big)\, ,\label{bounds}\\
&\Vert dT(\gamma')-dT(\gamma)\Vert_{{\cal B}(X)} 
\leq\, \alpha L_{\kappa,r} \Big(1+\alpha\Vert \gamma\,\vert  D\vert^{1/2}\Vert_{\sigma_1(\mathcal{H})}\Big)\Vert \gamma'-\gamma\Vert_X\; .
\label{bounds-lip}
\end{eqnarray}

\end{prop}

\begin{proof}
If $\gamma\in \Gamma_{\leq q}^r$ and $\gamma'\in\Gamma_{\leq q}$ then  the operator $\gamma'':=P^+_{V,\gamma} \gamma'P^+_{V,\gamma}$ is in $X$. Indeed, from \eqref{commutator} one has
$$\Vert \vert D \vert^{\frac{1}{2}}\gamma''\vert D \vert^{\frac{1}{2}}\Vert_{\sigma_1(\mathcal H)}\leq \big\Vert\, \vert  D \vert^{-\frac{1}{2}}P^+_{V,\gamma} \vert  D \vert^{\frac{1}{2}}\big\Vert_{{\cal B}({\cal H})}^2\Vert \vert D \vert^{\frac{1}{2}}\gamma'\vert D \vert^{\frac{1}{2}}\Vert_{\sigma_1(\mathcal H)}<\infty\,.$$
In addition, $\tr_{\mathcal H} \gamma''\leq \tr_{\mathcal H} \gamma'\leq q$ and $0\leq \gamma''\leq P^+_{V,\gamma}\leq  id_{\mathcal H}$, so $\gamma''$ is in $\Gamma_{\leq q}$.\medskip

\noindent
In the special case $\gamma=\gamma'\in \Gamma_{\leq q}$, this tells us that $T(\gamma)$ is in $\Gamma_{\leq q}$.\medskip

\noindent
In the general case, we may write
$$T(\gamma)-\gamma''=P^+_{V,\gamma} (\gamma-\gamma')P^+_{V,\gamma}\,,$$
hence
 $$\Vert T(\gamma)-\gamma''\Vert_{\sigma_1(\mathcal H)}\leq \Vert \gamma-\gamma'\Vert_{\sigma_1(\mathcal H)}\,,$$
so
$\hbox{dist}_{\sigma_1(\mathcal H)}(T(\gamma),\Gamma_{\leq q})\leq \hbox{dist}_{\sigma_1(\mathcal H)}(\gamma,\Gamma_{\leq q})<r$. This proves that $T(\gamma)\in \Gamma_{\leq q}^r$.\medskip

Now, we may write
\begin{align*}
T^2(\gamma)-T(\gamma)=&\ P^+_{V,T(\gamma)} T(\gamma) P^+_{V,T(\gamma)}  - P^+_{V,\gamma}T(\gamma) P^+_{V,\gamma}\\
= &\ \left(Q(T(\gamma))-Q(\gamma)\right) T(\gamma) + T(\gamma) \left(Q(T(\gamma))-Q(\gamma)\right)\\
&\ + \left(Q(T(\gamma))-Q(\gamma)\right) T(\gamma)\left(Q(T(\gamma))-Q(\gamma)\right)
\end{align*}
hence
\begin{align*}
\Vert T^2(\gamma)-T(\gamma)\Vert_X\leq&\ 2\left\Vert \vert  D\vert^{1/2} \left(Q(T(\gamma))-Q(\gamma)\right) T(\gamma)\vert  D\vert^{1/2} \right\Vert_{\sigma_1({\cal H})}\\
&\ + \left\Vert \vert  D\vert^{1/2} \left(Q(T(\gamma))-Q(\gamma)\right) T(\gamma)\left(Q(T(\gamma))-Q(\gamma)\right)\vert  D\vert^{1/2} \right\Vert_{\sigma_1({\cal H})} \\
\leq&\ 2 \Vert Q(T(\gamma))-Q(\gamma)\Vert_Y\Vert T(\gamma)\,\vert  D\vert^{1/2}\Vert_{\sigma_1({\cal H})}\\
&\ + \Vert Q(T(\gamma))-Q(\gamma)\Vert_{Y}^2\Vert T(\gamma)\Vert_{\sigma_1({\cal H})}\,.\end{align*}
But we have seen that $\Vert Q(T(\gamma))-Q(\gamma)\Vert_{Y}\leq a_r \Vert T(\gamma)-\gamma\Vert_X$ and $\Vert T(\gamma)\Vert_{\sigma_1({\cal H})}\leq q+r\,,$ so estimate \eqref{main-estimate} holds. \medskip

Now, from Lemma \ref{regularity}, $T$ is in $C^1(\Gamma_{\leq q}^r,X)$ with the following formula:
\[ dT(\gamma)h=(dQ(\gamma) h)\gamma P^+_{V,\gamma} + (adjoint) + P^+_{V,\gamma} h P^+_{V,\gamma}\,.\]
Using the inequality \eqref{commutator} of Lemma \ref{hardy}, we may write
\begin{align*}
\Vert (dQ(\gamma) h)\gamma P^+_{V,\gamma} \Vert_X &\leq \Vert dQ(\gamma) h \Vert_Y \Vert \gamma \vert  D\vert^{1/2}
\Vert_{\sigma_1({\cal H})}\left\Vert \vert D \vert^{-1/2}P^+_{V,\gamma} \vert D \vert^{1/2}\right\Vert_{{\cal B}({\cal H})}\\
&\leq a_r\Vert h\Vert_X\, \big\Vert \gamma \vert  D\vert^{1/2}
\big\Vert_{\sigma_1({\cal H})}\left(\frac{1+\kappa_r}{1-\kappa_r}\right)^{1/2}\,,
\end{align*}
\begin{align*}
\Vert P^+_{V,\gamma} h P^+_{V,\gamma}\Vert_X&\leq  
\left\Vert \vert D \vert^{-1/2}P^+_{V,\gamma} \vert D \vert^{1/2}\right\Vert^2_{{\cal B}({\cal H})}\Vert h \Vert_X\\
&\leq \left(\frac{1+\kappa_r}{1-\kappa_r}\right)\Vert h \Vert_X\,.
\end{align*}
Estimate \eqref{bounds} follows from these bounds. The proof of estimate \eqref{bounds-lip} is more tedious but goes along the same lines, so we omit the details: one just needs to estimate each term of the sum
\begin{align*}
\big(dT(\gamma')-dT(\gamma)\big)h=&\big\{\big((dQ(\gamma')-dQ(\gamma)) h\big)\gamma' P^+_{V,\gamma'} + 
(dQ(\gamma) h)(\gamma'-\gamma) P^+_{V,\gamma}\\
&+(dQ(\gamma) h)\gamma P^+_{V,\gamma}\big(Q(\gamma')-Q(\gamma)\big)\big\}
+\{adjoint\}\\
&+ \big(Q(\gamma')-Q(\gamma)\big) h P^+_{V,\gamma'}
+ P^+_{V,\gamma} h \big(Q(\gamma')-Q(\gamma)\big)\,.
\end{align*}
\end{proof}

We now define an open subset ${\cal U}$ of $\Gamma_{\leq q}^r$ allowing us to apply Proposition \ref{abstract-regular}.

\begin{prop}\label{def-U}

Assume that $\kappa_r < 1\,,$ $\alpha\, (Z+r)<\frac{2}{\pi/2+2/\pi}\,$ and take $a_r$ as in Lemma \ref{regularity}. Given $0<R<{1\over 2a_r}\,,$ let $A:=\max\left(\frac{2+a_r(q+r)}{2}\,,\,\frac{1}{1-2a_r R}\right)\,$ and
$${\cal U}:=\{ \gamma \in \Gamma_{\leq q}^r\;:\;\Vert\,\gamma\,\vert  D\vert^{1/2}\Vert_{\sigma_1({\cal H})}+A\Vert T(\gamma)-\gamma\Vert_X < R\,\}\,.$$

Then ${\cal U}$ satisfies the assumptions of Proposition \ref{abstract-regular} with $k:=2a_rR\,.$
\end{prop}

\begin{proof}

First of all, if $\gamma \in {\cal U}$, then
\begin{align*} \Vert T(\gamma)\,\vert  D\vert^{1/2}\Vert_{\sigma_1({\cal H})} \leq &\ \Vert \gamma\,\vert  D\vert^{1/2}\Vert_{\sigma_1({\cal H})} + \Vert (T(\gamma)-\gamma)\,\vert  D\vert^{1/2}\Vert_{\sigma_1({\cal H})} \\
\leq &\ \Vert \gamma\,\vert  D\vert^{1/2}\Vert_{\sigma_1({\cal H})} + \Vert T(\gamma)-\gamma\Vert_X \,,
\end{align*}
hence, using the inequality $A\geq \frac{2+a_r (q+r)}{2}\,,$
\begin{align*} \Vert T(\gamma)\,\vert  D\vert^{1/2}\Vert_{\sigma_1({\cal H})}+&\frac{a_r (q+r)}{2}\Vert T(\gamma)-\gamma\Vert_X \\
&\leq \Vert \gamma\,\vert  D\vert^{1/2}\Vert_{\sigma_1({\cal H})} +\frac{2+a_r (q+r)}{2}\Vert T(\gamma)-\gamma\Vert_X
< \ R\,.
\end{align*}
In addition, $T(\gamma)\in\Gamma_{\leq q}^r$ and \eqref{main-estimate} implies that 
$$\Vert T^2(\gamma)-T(\gamma)\Vert_X\leq k \Vert T(\gamma)-\gamma\Vert_X$$
with $k:=2a_r R<1$. Thus, using the inequality $A\geq \frac{1}{1-2a_r R}\,$ we get
\begin{align*} \Vert T(\gamma)\,\vert  D\vert^{1/2}\Vert_{\sigma_1({\cal H})}+&A\Vert T^2(\gamma)-T(\gamma)\Vert_X  \\
&\leq \Vert \gamma\,\vert  D\vert^{1/2}\Vert_{\sigma_1({\cal H})}+(1+Ak)\Vert T(\gamma)-\gamma\Vert_X
< R\,,
\end{align*}
so $T(\gamma)\in{\cal U}\,.$\medskip

Finally, from the definition of $\mathcal U$ we immediately see that $\sup_{\gamma \in{\cal U}} \Vert T(\gamma)-\gamma\Vert_X$ is finite. Moreover, \eqref{bounds} implies that $\sup_{\gamma \in{\cal U}} \Vert dT(\gamma)\Vert_X<\infty$
and \eqref{bounds-lip}  implies that $dT$ is Lipschitzian on ${\cal U}\,.$ This ends the proof of Proposition \ref{def-U}.

\end{proof}

We are now ready to state the main result of this subsection:

\begin{thm}\label{notre-theta}

Assume that $\kappa_r < 1\,$ and $\,\alpha\, (Z+r)<\frac{2}{\pi/2+2/\pi}\,$. Let $a_r$ be as in Lemma \ref{regularity} and $R<{1\over 2a_r}\,.$ Let ${\cal U}$ and $k$ be as in Proposition \ref{def-U} and let $\overline{\cal U}$ be the closure of ${\cal U}$ in $X$. Then the sequence of iterated maps $(T^p)_{p\geq 0}$ converges
uniformly on $\overline{\cal U}$ to a limit $\theta$ with $\theta(\overline{\cal U})\subset {\rm Fix}(T)\cap \overline{\cal U}\,$ and ${\rm Fix}(\theta)={\rm Fix}(T)\cap \overline{\cal U}$. We have the estimate

$$\forall \gamma\in \overline{\cal U}\,,\; \Vert \theta(\gamma)-T^p(\gamma)\Vert_X\,\leq\,{k^p\over 1-k}\,\Vert T(\gamma)-\gamma\Vert_X\,.$$

Moreover $\theta\in C^{1,\rm{unif}}({\cal U}, X)$ and $d(T^p)$ converges uniformly to $d\theta$ on ${\cal U}$.\medskip

In this way we obtain a retraction $\theta$ of $\,\overline{\cal U}$ onto ${\rm Fix}(T)\cap \overline{\cal U}$ whose restriction to $\,{\cal U}$ is of class $C^{1,\rm{unif}}$. More precisely, $id_{\cal U}-\theta$ and its differential are bounded and uniformly continuous on $\,{\cal U}\,.$\medskip

For any $\gamma\in {\rm Fix}(T)\cap {\cal U}\,$ and any $h\in X\,,$ the operator $S=d\theta(\gamma)\,h$ satisfies
$$P^+_{V,\gamma} S P^+_{V,\gamma}=P^+_{V,\gamma} h P^+_{V,\gamma}\quad{\rm and}\quad P^-_\gamma S P^-_\gamma=0\;.$$
In other words, the splitting ${\cal H}=P^+_{V,\gamma}{\cal H}\oplus P^-_{V,\gamma}{\cal H}$ gives a block decomposition of $d\theta(\gamma)\,h$ of the form
\begin{equation}\label{blocks}
d\theta(\gamma)\,h=\left(\begin{matrix}
P^+_{V,\gamma} h P^+_{V,\gamma} &  b_{\gamma}(h)^*\\ b_{\gamma}(h) & 0\\
\end{matrix}\right)
\end{equation}

\end{thm}

\begin{proof}
The existence and regularity properties of the retraction $\theta$ follow from Proposition \ref{abstract-regular} and Proposition \ref{def-U}. To end the proof of Theorem \ref{notre-theta}, it suffices to show \eqref{blocks}.\medskip

\noindent
We recall that for any $\gamma\in \Gamma_{\leq q}^r$ and $h\in X$, 
\[dT(\gamma)h=(dQ(\gamma)h) \gamma P^+_{V,\gamma} + (adjoint)+P^+_{V,\gamma} h P^+_{V,\gamma}\,. \]
Multiplying this formula from both sides by $P^-_{V,\gamma}$, we get
\[ P^-_{V,\gamma}(dT(\gamma)h)P^-_{V,\gamma}=0\,.\]
On the other hand, we have $P^+_{V,\gamma}P^-_{V,\gamma}=0$. Differentiating this identity, we find
\[ (dQ(\gamma)h)P^-_{V,\gamma} - P^+_{V,\gamma}(dQ(\gamma)h)=0\,.\]
Multiplying this formula from the right by $P^+_{V,\gamma}$ we get
\[ P^+_{V,\gamma}(dQ(\gamma)h)P^+_{V,\gamma}=0\,.
\]
But for $\gamma\in {\rm Fix}(T)\cap {\cal U}\,$ the formula for $dT(\gamma)$ can be rewritten in the form
\[dT(\gamma)h=(dQ(\gamma)h) P^+_{V,\gamma}\gamma P^+_{V,\gamma} + (adjoint)+P^+_{V,\gamma} h P^+_{V,\gamma}\,. \]
Multiplying this formula from both sides by $P^+_{V,\gamma}$, we get
\[P^+_{V,\gamma}(dT(\gamma)h)P^+_{V,\gamma}= P^+_{V,\gamma} h P^+_{V,\gamma}\,.\]
Moreover, since $T(\gamma)=\gamma$, for any integer $p\geq 1$ we have
\[
    d(T^{p})(\gamma)h= dT(\gamma) \big( d(T^{p-1})(\gamma)h\big)\,.
\]
So we immediately get
\[ P^-_{V,\gamma}(d(T^p)(\gamma)h)P^-_{V,\gamma}=0\]
and we easily prove that
\begin{align*}
    P^+_{V,\gamma}(d(T^{p})(\gamma)h)P^+_{V,\gamma}= P^+_{V,\gamma} h P^+_{V,\gamma}
\end{align*}
by induction on $p$.\medskip

\noindent
Passing to the limit $p\to +\infty$ we conclude that
\begin{align*}
    P^-_{V,\gamma}(d\theta(\gamma) h)P^-_{V,\gamma}=0\quad \textrm{and} \quad P^+_{V,\gamma}(d\theta(\gamma) h)P^+_{V,\gamma}=P^+_{V,\gamma}h P^+_{V,\gamma}\,.
\end{align*}
This proves \eqref{blocks}.
\end{proof}

Since $\mathcal U$ and $\Gamma_{\leq q}$ are both invariant under $T$, it is natural to consider their intersection $\mathcal V$. The set $\mathcal V:=\mathcal U\cap \Gamma_{\leq q}$ is relatively open in $\Gamma_{\leq q}$ and invariant under $T$. Its closure $\overline{\mathcal V}$ is invariant under $\theta$, and the restriction of $\theta$ to $\overline{\mathcal V}$ is the retraction of $\overline{\mathcal V}$ onto $\Gamma^+_{\leq q}\cap \overline{\mathcal V}$ announced in the introduction. In the sequel we shall only need to work with this restriction.
The last question we address in this section is whether the sublevel set $\{\gamma\in\Gamma^+_{\leq q}\,:\;{\E}_{DF}(\gamma)-\tr_{\mathcal H}(\gamma)\leq 0\}$ is included in ${\cal V}$. To answer it positively, we need the following result:

\begin{prop}\label{estim-fixed}
Assume that $\kappa < 1-{\pi\over 4}\alpha\,q\,$. Let $\gamma\in\Gamma_{\leq q}^+\,$ be such that
$${\cal E}_{DF}(\gamma)-\tr_{\mathcal H}(\gamma)\leq 0\,.$$
Then
$$\Vert \gamma \Vert_X \leq \Big( 1-\kappa-{\pi\over 4}\alpha\,q\Big)^{-1}q\,.$$
\end{prop}

\begin{proof}
Let $\gamma\in \Gamma_{\leq q}^+$ such that $\mathcal{E}_{DF}(\gamma)-\tr_{\mathcal H}(\gamma)\leq 0$. As $D_{\gamma}\gamma=|D_{\gamma}|\gamma$ and from Lemma \ref{hardy}, we have
\begin{align*}
    \mathcal{E}_{DF}(\gamma)-\tr_{\mathcal H}(\gamma)&=\tr[(D_{\gamma}-1-\frac{\alpha}{2}W_{\gamma})\gamma]\\
    &=\tr[(|D_{\gamma}|-1-\frac{\alpha}{2}W_{\gamma})\gamma]\\
    &\geq (1-\kappa-\frac{\pi}{4}\alpha q)\|\gamma\|_X-\|\gamma\|_{\sigma_1(\mathcal{H})},
\end{align*}
hence,
\begin{align*}
    \|\gamma\|_X\leq (1-\kappa-\frac{\pi}{4}\alpha q)^{-1}[\mathcal{E}_{DF}(\gamma)-\tr_{\mathcal H}(\gamma)+q]\leq (1-\kappa-\frac{\pi}{4}\alpha q)^{-1}q.
\end{align*}
\end{proof}

We recall that the construction of $\mathcal U$ given in Proposition \ref{def-U} involves a parameter $R\in (0,\frac{1}{2a_r})$ and that $\mathcal V=\mathcal U\cap \Gamma_{\leq q}$. Proposition \ref{estim-fixed} has the following consequence:

\begin{corollary}\label{nonempty}
Assume that $\kappa<1-{\pi\over 4}\alpha\,q\,$, $\alpha\,(Z+r)<\frac{2}{\pi/2+2/\pi}\,$ and that
\begin{equation}\label{condition with r}
\pi\alpha q \,< \,2(1-\kappa_r)^{1\over 2}\lambda_r^{1\over 2}\Big( 1-\kappa-{\pi\over 4}\alpha\,q\Big)^{1\over 2}\,.
\end{equation}
Then one can choose $0<R<\frac{1}{2a_r}$ and $\rho>0$ such that, for all $\gamma\in\Gamma^+_{\leq q}\,$ satisfying ${\cal E}_{DF}(\gamma)-\tr_{\mathcal H}(\gamma)\leq 0\,,$ there holds $B_X(\gamma,\rho)\cap \Gamma_{\leq q}\subset {\cal V}$.
\end{corollary}

\begin{proof}
Proposition \ref{estim-fixed} implies that $ \|\gamma\|_X\leq (1-\kappa-\frac{\pi}{4}\alpha q)^{-1}q\,.$ So, by Cauchy-Schwarz,
$$\Vert\,\gamma\,\vert  D\vert^{1/2}\Vert_{\sigma_1({\cal H})}\leq \Vert\,\gamma\,\Vert_{_X}^{1/2}\Vert\,\gamma\,\Vert_{\sigma_1({\cal H})}^{1/2}\leq (1-\kappa-\frac{\pi}{4}\alpha q)^{-1/2}q\,.$$
Now, condition \eqref{condition with r} tells us that $(1-\kappa-\frac{\pi}{4}\alpha q)^{-1/2}q < \frac{1}{2a_r}$. Moreover, since $\gamma\in\Gamma^+_{\leq q}$ one has $\Vert T(\gamma)-\gamma\Vert_X=0\,$. So, taking $\rho$ and $\frac{1}{2a_r}-R$ positive and small enough, using \eqref{bounds} one finds that for any $\gamma'\in B_X(\gamma,\rho)\cap \Gamma_{\leq q}$,
$$\Vert\,\gamma'\,\vert  D\vert^{1/2}\Vert_{\sigma_1({\cal H})} + A\Vert T(\gamma')-\gamma'\Vert_X < R\,,$$
where $A$ is the same as in Proposition \ref{def-U}. This inequality means that $\gamma'\in \mathcal V$.
\end{proof}

\section{Existence of a ground state.}\label{existence}

The first result of this section is
\begin{prop}\label{minseq} If $Z$, $q$, $1-\kappa$ are positive and $\alpha\,Z<\frac{2}{\pi/2+2/\pi}\,$, then $E_q<0\,$.
\end{prop}
\begin{proof}
The self-adjoint operator $D+V$ has infinitely many eigenvalues in the interval $(0,1)$ (see e.g. \cite{ELS2}). As a consequence, we can find a projector $\Pi$ of rank $1$ such that $\Pi\leq \1_{\{\mu\}}( D+V)$ for some $0<\mu<1\,$. Taking $\varepsilon>0$ small enough, we get $\varepsilon \Pi\in {\cal V},\,\theta(\varepsilon \Pi)\in \Gamma^+_{\leq q}$ and ${\E}_{DF}(\theta(\varepsilon \Pi))-\tr_{\mathcal H}(\theta(\varepsilon \Pi))=(\mu-1)\varepsilon+o(\varepsilon)<0$, so the infimum of ${\E}_{DF}-\tr_{\mathcal H}\,$ on $\Gamma^+_{\leq q}$ is negative.\end{proof}

In order to prove Theorem \ref{main} we have to study the convergence of minimizing sequences for $ {\cal E}_{DF}-\tr_{\mathcal H}$. We need the following estimate related to the mean-field operator:

\begin{lemma}\label{estim-proj}
Assume that $ \kappa < 1\,.$ Let $\tilde{\gamma}\in\Gamma_{\leq q}$ and let $\gamma\in \sigma_1(\mathcal H)$ be such that $0\leq\gamma\leq \1_{[0,\nu]}( D_{V,\tilde \gamma})$ for some $\nu>0$. Then $ D\, \gamma \,  D\in \sigma_1(\mathcal H)$ and the following estimate holds:
$$\Vert \, D\, \gamma \,  D\, \Vert_{\sigma_1({\cal H})} \leq (1-\kappa)^{-2} \nu^2\tr_{\mathcal H} (\gamma)\,.$$
\end{lemma}

\begin{proof}
By assumption $\gamma=\1_{[0,\nu]}( D_{V,\tilde \gamma})\,\gamma\,\1_{[0,\nu]}( D_{V,\tilde \gamma})$ and $\tr_{\mathcal H} (\gamma) = \Vert \gamma\Vert_{\sigma_1({\cal H})}$, so
$$\Vert \, D_{V,\tilde{\gamma}}\,\gamma\,  D_{V,\tilde{\gamma}}\,\Vert_{\sigma_1(\mathcal H)}\leq
\Vert  D_{V,\tilde{\gamma}}\1_{[0,\nu]}( D_{V,\tilde \gamma})\Vert^2_{\mathcal B(\mathcal H)} \Vert \gamma\Vert_{\sigma_1({\cal H})}\leq \nu^2\tr_{\mathcal H} (\gamma)\,.$$
Then, using \eqref{hardy3} for $s=1$, one gets
$$\Vert \, D\, \gamma \,  D\, \Vert_{\sigma_1({\cal H})} \leq \Vert\, D \,  D_{V,\tilde{\gamma}}^{-1} \,\Vert_{{\cal B}({\cal H})}^2 \Vert \, D_{V,\tilde{\gamma}}\,\gamma\, D_{V,\tilde{\gamma}}\,\Vert_{\sigma_1(\mathcal H)}\leq (1-\kappa)^{-2} \nu^2\tr_{\mathcal H} (\gamma)\,.$$
\end{proof}

The next lemma of this section gives a crucial property of minimizing sequences: their terms are approximate ground states of their mean-field Hamiltonian.

\begin{lemma}\label{minimizing}

Assume that $Z$, $q$, $1-\kappa-\frac{\pi}{4}q$ are positive and that \eqref{condition} is satisfied. Let $(\gamma_n)$ be a minimizing sequence for $ {\cal E}_{DF}-\tr_{\mathcal H}$ in $\Gamma^+_{\leq q}\,.$ Then
$$\lim\limits_{n\to\infty}\bigg\{{\rm tr}\big(( D_{V,\gamma_n}-1)\gamma_n\big)\,-\inf_{\gamma\in\Gamma_{\leq q}\,,\;\gamma =  P^+_{\gamma_n}\gamma}{\rm tr}\big(( D_{V,\gamma_n}-1)\gamma \big)\bigg\}= 0\;.$$

\end{lemma}

\begin{proof}

The proof of this lemma is based on the construction of Section \ref{retraction}. We take $r>0$ such that the assumptions of Corollary \ref{nonempty} are satisfied and we choose $R,\,\rho$ as in this corollary. As a consequence of Proposition \ref{minseq}, for $n$ large enough we have ${\E}_{DF}(\gamma_n)-\tr_{\mathcal H}(\gamma_n) <0\,$, so Proposition \ref{estim-fixed} gives us a bound on $\Vert\gamma_n\Vert_X$ and Corollary \ref{nonempty} implies that $\Gamma_{\leq q}\cap B_X(\gamma_n,\rho)\subset \mathcal V$.\medskip

Assume by contradiction that the minimizing sequence $(\gamma_n)$ does not satisfy the conclusion of the lemma. Then there is $\varepsilon_0>0$ such that, after extraction,
$${\rm tr}\big(( D_{V,\gamma_n}-1)\gamma_n\big)\geq \inf_{\gamma\in\Gamma_{\leq q}\,,\;\gamma =  P^+_{\gamma_n}\gamma}{\rm tr}\big(( D_{V,\gamma_n}-1)\gamma \big)+\varepsilon_0\,.$$

On the other hand, for each $\nu>1$ there exists a sequence $(g_n)$ of bounded self-adjoint operators of rank $q$ such that $0\leq g_n\leq \1_{[0,\nu]}( D_{V,\gamma_n})$ and
$${\rm tr}\big(( D_{V,\gamma_n}-1)g_n\big)\leq\inf_{\gamma\in\Gamma_{\leq q}\,,\;\gamma =  P^+_{\gamma_n}\gamma}{\rm tr}\big(( D_{V,\gamma_n}-1)\gamma \big)+{\varepsilon_0 \over 2}\,.$$

Taking for instance $\nu= 2$, from Lemma \ref{estim-proj} we get a bound on $\Vert g_n\Vert_X\,.$ As a consequence, there is $\sigma>0$ such that for any $s\in[0,\sigma]\,,$ the convex combination $(1-s)\,\gamma_n+s\,g_n\,$ is in $\,\Gamma_{\leq q}\cap B_X(\gamma_n,\rho)\,$, so, from Corollary \ref{nonempty}, it lies in $\mathcal V$ when $n$ is large enough. Thus, from Theorem \ref{notre-theta}, the function
\[f_n\,:\,s\in[0,\sigma]\mapsto \big({\cal E}_{DF}-\tr_{\mathcal H}\big)\bigl(\theta[(1-s)\,\gamma_n+s\,g_n]\bigr)\]
is well-defined and of class $C^1$. Moreover, the sequence of derivatives $(f'_n)$ is equicontinuous on $[0,\sigma]$. From Formula (\ref{blocks}),
 $$f'_n(0)={\rm tr}\big(( D_{V,\gamma_n}-1)(g_n-\gamma_n)\big)\leq - {\varepsilon_0 \over 2}\,,$$
 so there is $0<s_0<\sigma$ independent of $n\,,$ such that $\forall s\in [0,s_0]\;,\;f'_n(s)\leq -{\varepsilon_0\over 4}\,.$ Hence
 $$\big({\cal E}_{DF}-\tr_{\mathcal H}\big)\bigl(\theta[(1-s_0)\gamma_n+s_0 g_n]\bigr)=f_n(s_0)\leq f_n(0)-{\varepsilon_0 s_0\over 4}=\big({\cal E}_{DF}-\tr_{\mathcal H}\big)(\gamma_n)-{\varepsilon_0 s_0\over 4}\,.$$
 But $\theta[(1-s_0)\gamma_n+s_0 g_n]\in \Gamma_{\leq q}^+$ and $\big({\cal E}_{DF}-\tr_{\mathcal H}\big)(\gamma_n) \to E_q\,.$ This is a contradiction. So Lemma \ref{minimizing} is proved.

\end{proof}

It turns out that the compactness of minimizing sequences is easier to study for positive ions. So in order to prove Theorem \ref{main} we are going to start with the case $q<Z$, which is contained in the following proposition:

\begin{prop}\label{perturb}

Consider the Dirac-Fock problem with $0<q<Z\,.$
Assume that $\kappa < 1-{\pi\over 4}\alpha\,q\,$ and that condition (\ref{condition}) is satisfied. Then there exists $\gamma_*\in \Gamma_{q}^+$ such that
$$ {\cal E}_{DF}(\gamma_*)-\tr_{\mathcal H}(\gamma_*)=E_q\,.$$
Any such ground state may be written $\gamma_*=\1_{(0,\mu)}( D_{V,\gamma_*})+\delta\,$ with
$\, 0\leq \delta \leq \1_{\{\mu\}}( D_{V,\gamma_*})$
for some $\mu\in (0,1)$.\medskip

\noindent
Moreover, for $h>0$ and small, one has $E_{q+h} < E_q$.

\end{prop}

It follows directly from the definition of $E_q$ that the function $q\mapsto E_q$ is nonincreasing, so the last statement of Proposition \ref{perturb} directly implies the strict binding inequalities for positive ions and neutral atoms:

 \begin{corollary}\label{atoms}
 Consider the Dirac-Fock problem with $0<q\leq Z\,.$
Assume that $\kappa < 1-{\pi\over 4}\alpha\,q\,$ and that condition (\ref{condition}) is satisfied. Then the map $r\mapsto E_r$ is strictly decreasing on $[0,q]$, so the strict binding inequalities \eqref{cc} hold.
 \end{corollary}

 In our proof of Proposition \ref{perturb}, a crucial tool will be a uniform estimate on the spectrum of the operators $ D_{V,\gamma}\,$. If $\lceil q\rceil$ denotes the smallest integer larger or equal to $q$, this estimate is given in the following lemma:
\begin{lemma} \label{spectrum}
Assume that $\alpha\, Z<\frac{2}{\pi/2+2/\pi}\,$ and that $0<q<Z$. Then:
\medskip

$\bullet$ There is a constant $e\in (0,1)$ such that for any $\gamma\in \Gamma_{\leq q}\,$, the mean-field operator $ D_{V,\gamma}$ has at least $\lceil q\rceil$ eigenvalues (counted with multiplicity) in the interval $[0,1-e]$.\medskip

$\bullet$  There is a nonnegative integer $N$ such that for any $\gamma\in \Gamma_{\leq q}\,$, the mean-field operator $ D_{V,\gamma}$ has at most $\lceil q\rceil+N$ eigenvalues (counted with multiplicity) in $[0,1-\frac{e}{2}]$.
\end{lemma}

\begin{proof}

For the first statement of the lemma, the arguments are similar to the proof of Lemma 4.6 in \cite{ES1}, with some necessary adaptations. One takes a subspace $S$ of $C^\infty_c(\,(0,\infty);\R)$ of dimension $\lceil q\rceil$. 
Given $t>1$ we call $G_{t}$ the
$\lceil q\rceil$-dimensional subspace of $C^\infty_c(\R^3;\C^4)$ consisting of all functions $\psi$ of the form
\begin{equation}\label{dilation}
\psi (x) \ = \ \left(\begin{array}{cccc} f(|x|/ t) \\ 0 \\ 0 \\ 0 \\ \end{array} \right) \;,\quad f\in S\,
. \end{equation}

One easily finds two constants $0 < c_\ast< c^\ast < \infty$ such that,
for any $t>1$ and $\psi \in G_{t}\,$,
\begin{eqnarray} \big\langle\Lambda^+\psi, \sqrt{1-\Delta}\,\Lambda^+\psi\big\rangle_{L^2} & \leq & \bigl(1+\frac{c^\ast}{t^2}\bigr)
\ \bigl \Vert \psi \bigr \Vert^2_{L^2}\,, \label{free}\\ \bigl \Vert \nabla \psi \bigr
\Vert^2_{L^2} &
\leq &
\frac{c^\ast}{t^2}
\ \bigl \Vert \psi \bigr \Vert^2_{L^2}\,, \label{kinetic}\\
\left\langle \psi, \frac{1}{\vert\cdot\vert} \psi \right\rangle_{L^2} & \geq &
\frac{c_\ast}{t}
\ \bigl \Vert \psi \bigr \Vert^2_{L^2}\,,\\
\bigl \Vert \Lambda^- \psi \bigr
\Vert_{L^2} & \leq & \frac{c^\ast}{t}
\ \bigl \Vert \psi \bigr \Vert_{L^2}\,, \label{lambda}\\
\bigl \Vert \nabla(\Lambda^- \psi) \bigr
\Vert_{L^2} & \leq & \frac{c^\ast}{t^2}
\ \bigl \Vert \psi \bigr \Vert_{L^2}\,, \\
\qquad\quad\big\langle\psi, V\psi\big\rangle_{L^2} & \leq &  -\alpha Z\left\langle\psi,\frac{1}{\vert\cdot\vert}
\psi\right\rangle_{L^2}+o\left(\frac{1}{t}\right)_{t\to\infty}||\psi||_{_{L^2}}^2
 \,.\\
\nonumber \end{eqnarray}  

Now, we recall that for $\gamma\in\Gamma_{\leq q}$ one has $W_\gamma\leq \rho_\gamma * \frac{1}{\vert\cdot\vert}$. Moreover,
since $\psi$ in $G_{t}$ is radial, one has $\Bigl\langle\psi, \rho_\gamma * \frac{1}{\vert\cdot\vert}\psi\Bigr\rangle_{L^2} \leq \Bigl\langle \psi, \frac{q}{\vert\cdot\vert} \psi \Bigr\rangle_{L^2}$, so that, for $t$ large enough:
\begin{equation}
\Bigl\langle\psi,(V+\alpha W_\gamma)\psi\Bigr\rangle_{L^2} \leq -\alpha (Z-q)\frac{c_*}{2 t}
\ \bigl \Vert \psi \bigr \Vert^2_{L^2}\,.
\end{equation}
On the other hand, $\Vert(V+\alpha W_\gamma)\Lambda^-\psi\Vert_{L^2} \leq 2\alpha(Z+q)\bigl \Vert \nabla(\Lambda^- \psi) \bigr
\Vert_{L^2}$, so, for $t$ large:
\begin{equation}\label{potential}
\Bigl\langle\Lambda^+\psi,(V+\alpha W_\gamma)\Lambda^+\psi\Bigr\rangle_{L^2} \leq -\alpha (Z-q)\frac{c_*}{4 t}
\ \bigl \Vert \Lambda^+\psi \bigr \Vert^2_{L^2}\,.
\end{equation}
For $\psi_+\in\Lambda^+C^\infty_c(\R^3,\C^4)$ and $0<e<1$, we define
\begin{align*}
\mathcal Q_{1-e}(\psi^+)&:=\big\langle\psi^+, \vert D\vert\,\psi^+\big\rangle_{L^2}\\
&+\Bigl\langle\psi^+,(V+\alpha W_\gamma-1+e)\psi^+\Bigr\rangle_{L^2} \\
&+
\Bigl\langle(V+\alpha W_\gamma)\psi^+,\,\Lambda^-(\vert D\vert-V-\alpha W_\gamma+1-e)^{-1}\Lambda^-(V+\alpha W_\gamma)\psi^+\Bigr\rangle_{L^2}\,.
\end{align*}
Combining the estimates \eqref{free}, \eqref{kinetic}, \eqref{lambda} and \eqref{potential} one finds $\underbar{t}>1$ and $\underbar{c}>0$ such that for all $e\in (0,1)$, $t\geq\underbar{t}$, $\gamma\in\Gamma_{\leq q}$  and for every $\psi^+$ in the $\lceil q\rceil$-dimensional complex vector space $G_t^+:=\Lambda^+G_t\,$:
$$
\mathcal Q_{1-e}(\psi^+)< \left(e-\frac{\underbar{c}}{t}\right)
\ \bigl \Vert \psi^+ \bigr \Vert^2_{L^2}\,.
$$

From now on, we fix $t=\underbar{t}$ and $e=\frac{\underbar{c}}{2\underbar{t}}$. Then the above inequality tells us that the quadratic form $\mathcal Q_{1-e}$ is negative on $G^+_{\underbar{t}}$. Applying the abstract min-max theorem of \cite{DES} to the self-adjoint operator $ D_{V,\gamma}$ with the splitting of $\mathcal H$ associated to the free projectors $\Lambda^\pm=P^\pm_{0,0}$, we thus conclude that there are at least $\lceil q\rceil$ eigenvalues of $ D_{V,\gamma}$ (counted with multiplicity) in the interval $(0,1-e)$. Indeed, for $\psi^+$ in $ \Lambda^+ C^\infty_c(\R^3,\C^4)$ one has
\[ \mathcal Q_{1-e}(\psi^+) = \underset{\psi^-\in \Lambda^- C^\infty_c(\R^3,\C^4)}{\sup}\Bigl\{ \Bigl\langle\psi^++\psi^-,  D_{V,\gamma}(\psi^++\psi^-)\Bigr\rangle_{L^2}
-(1-e)\Vert \psi^++\psi^-\Vert^2_{L^2}\Bigr\}\,.\]
So, if $\lambda_{k,\gamma}$ denotes the $k$-th positive eigenvalue of $D_{V,\gamma}$ counted with multiplicity, from \cite{DES} we find that
\[1-e\, \geq\, \underset{\underset{{\rm dim} \mathcal W = \lceil q\rceil}{\mathcal W \rm{\;subspace\;of\;} \Lambda^+ C^\infty_c}}{\inf}\ \  \underset{\psi\in (\mathcal W\oplus \Lambda^- C^\infty_c)\setminus\{0\}}{\sup}
\ \frac{\langle\psi,  D_{V,\gamma}\psi\rangle_{L^2}}{\Vert \psi\Vert^2_{L^2}}=\lambda_{\lceil q\rceil,\gamma}\,.
\]
The first statement of the lemma is thus proved.\medskip

The second statement is easier. We notice that $ D_{V,\gamma}\geq  D_{V,0}$, so, invoking once again the min-max principle of \cite{DES}, we see that $\lambda_{k,\gamma}\geq\lambda_{k,0}$. Moreover the essential spectrum of $ D_{V,0}$ is $\R\setminus (-1,1)$, so $\lim_{k\to\infty} \lambda_{k,0}=1\,.$ Taking $N\geq 0$ such that $\lambda_{\lceil q\rceil+N+1,\,0} > 1-e/2$, we conclude that for any $\gamma\in\Gamma_{\leq q}$ there are at most $\lceil q\rceil+N$ eigenvalues of $D_{V,\gamma}$ in the interval $[0,1-e/2]$ and the lemma is proved.
\end{proof}

\noindent
Thanks to Lemma \ref{spectrum}, we can obtain more information on minimizing sequences:

\begin{lemma}\label{apriori}

Consider the Dirac-Fock problem with $0<q<Z\,$.
Assume that $\kappa < 1-{\pi\over 4}\alpha\,q\,$ and that condition (\ref{condition}) is satisfied. Let $(\gamma_n)$ be a minimizing sequence for $ ({\cal E}_{DF}-{\rm tr}_{\mathcal H})$ in $\Gamma^+_{\leq q}\,.$ For each $n$ define $p_n:=\1_{[0,1-\frac{e}{2}]}( D_{V,\gamma_n})\,$ where $e$ is given in Lemma \ref{spectrum}. Then
$$\tr_{\mathcal H}(\gamma_n)\to q\quad{\rm and}\quad\Vert \gamma_n - p_n\gamma_n p_n \Vert_X\to 0\;.$$
\end{lemma}

\begin{proof}

Let $\mu_n\in (0,1-e]$ be such that there are less than $\lceil q\rceil$ eigenvalues of $ D_{V,\gamma_n}$ (counted with their multiplicity) in the interval $[0,\mu_n)$ and at least $\lceil q\rceil$ in the interval $[0,\mu_n]\,.$ Then
$$\inf_{\gamma\in\Gamma_{\leq q}\,,\;\gamma =  P^+_{\gamma_n}\gamma}{\rm tr}\big(( D_{V,\gamma_n}-1)\gamma \big)=
{\rm tr}\big(( D_{V,\gamma_n}-\mu_n)\1_{[0,\mu_n)}( D_{V,\gamma_n})\big)+(\mu_n-1)q\;.$$

We define $p'_n:=\1_{(1-\frac{e}{2},\infty)}( D_{V,\gamma_n})\,.$
Since $\gamma_n=T(\gamma_n)$ we may write
\[{\rm tr}\big(( D_{V,\gamma_n}-\mu_n)\gamma_n\big)\,=\,{\rm tr}\big(( D_{V,\gamma_n}-\mu_n)p_n\gamma_n p_n\big)+{\rm tr}\big(( D_{V,\gamma_n}-\mu_n)p'_n\gamma_n p'_n\big)\,,\]
hence
\begin{eqnarray*}
&&{\rm tr}\big(( D_{V,\gamma_n}-1)\gamma_n\big)\,-\inf_{\gamma\in\Gamma_{\leq q}\,,\;\gamma  =  P^+_{\gamma_n}\gamma}{\rm tr}\big(( D_{V,\gamma_n}-1)\gamma \big) \\
&&=\,{\rm tr}\big(( D_{V,\gamma_n}-\mu_n)p'_n\gamma_n p'_n\big)
\\ && \;+\,{\rm tr}\big[( D_{V,\gamma_n}-\mu_n)\bigl(p_n\gamma_n p_n-\1_{[0,\mu_n)}( D_{V,\gamma_n})\bigr)\big] +(1-\mu_n)\big(q-\tr_{\mathcal H}(\gamma_n)\big)\;.
\end{eqnarray*}
Moreover $\,{\rm tr}\big(( D_{V,\gamma_n}-\mu_n)p'_n\gamma_n p'_n\big)$, ${\rm tr}\big[( D_{V,\gamma_n}-\mu_n)\bigl(p_n\gamma_n p_n-\1_{[0,\mu_n)}( D_{V,\gamma_n})\bigr)\big]$ and $(1-\mu_n)\big(q-\tr_{\mathcal H}(\gamma_n)\big)$ are nonnegative, so Lemma \ref{minimizing} implies that
\[\tr_{\mathcal H}(\gamma_n)\to q\,\ \hbox{ and }
\ {\rm tr}\big(( D_{V,\gamma_n}-\mu_n)p'_n\gamma_n p'_n\big)\to 0\;.\]

But $p'_n ( D_{V,\gamma_n}-\mu_n) p'_n\geq \frac{e}{2} p'_n$ and $\, p'_n ( D_{V,\gamma_n}-\mu_n)p'_n\geq p'_n\bigl(\vert  D_{V,\gamma_n}\vert-1+e\bigr) p'_n\,$\break so that, taking a convex combination of these two estimates:
$$p'_n ( D_{V,\gamma_n}-\mu_n)p'_n\geq \frac{e}{2-e}\, p'_n \vert  D_{V,\gamma_n}\vert p'_n\,,$$
hence $\Vert p'_n\gamma_n p'_n\Vert_X ={\rm tr}\big( p'_n\vert  D\vert p'_n\gamma_n\big)\leq (1-\kappa)^{-1}{\rm tr}\big(p'_n\vert  D_{V,\gamma_n}\vert p'_n\gamma_n\big)\to 0\,.$\medskip

It remains to study the limit of $u_n:=p'_n \gamma_n p_n$ as $n$ goes to infinity. Since $(\gamma_n)^2\leq \gamma_n\,,$ we have
$(p'_n \gamma_n p'_n)^2+u_n\,u_n^*= p'_n(\gamma_n)^2 p'_n\leq p'_n \gamma_n p'_n\,$,
hence
\[{\rm tr}\big(\vert  D_{V,\gamma_n}\vert^{1/2}u_n\,u_n^*\vert  D_{V,\gamma_n}\vert^{1/2}\big) \to 0\;.\]
Now, take $B\in {\cal B}({\cal H})$. By the Cauchy-Schwarz inequality,
\begin{eqnarray*}
&&{\rm tr}\big(B\,\vert  D_{V,\gamma_n}\vert^{1/2}u_n^*\vert  D_{V,\gamma_n}\vert^{1/2}\big) \\
&&\leq
\big[{\rm tr}\big(\vert  D_{V,\gamma_n}\vert^{1/2}p_nB^*B\, p_n\vert  D_{V,\gamma_n}\vert^{1/2}\big)\big]^{1/2}
\big[{\rm tr}\big(\vert  D_{V,\gamma_n}\vert^{1/2}u_n\,u_n^*\vert  D_{V,\gamma_n}\vert^{1/2}\big)\big]^{1/2}
\;.
\end{eqnarray*}
But $p_n$ has rank at most $\lceil q\rceil+N$ and $\big\Vert \,p_n\vert  D_{V,\gamma_n}\vert^{1/2}\big\Vert_{{\cal B}({\cal H})}\leq 1\;.$ As a consequence,
$${\rm tr}\big(\vert  D_{V,\gamma_n}\vert^{1/2}p_n B^*B\, p_n\vert  D_{V,\gamma_n}\vert^{1/2}\big) \leq (\lceil q\rceil+N) \Vert B \Vert^2_{{\cal B}({\cal H})}\;.$$
Since $B$ is arbitrary, this shows that $\big\Vert \,\vert  D_{V,\gamma_n}\vert^{1/2}u_n\vert  D_{V,\gamma_n}\vert^{1/2}\big\Vert_{\sigma_1({\cal H})}\to 0\,,$ hence\break $\Vert u_n\Vert_X\to 0\,.$ 
\medskip

Finally $\Vert \gamma_n-p_n\gamma_n p_n\Vert_X\leq \Vert p'_n\gamma_n p'_n\Vert_X+2\Vert u_n\Vert_X\to 0\,$ and the lemma is proved.
\end{proof}

Now we have

\begin{corollary}\label{limit}

With the same assumptions and notations as in Lemma \ref{apriori}, there exists $\gamma_*\in \Gamma_{\leq q}$ such that, after extraction of a subsequence, $\Vert \gamma_n-\gamma_*\Vert_X \to 0$ as $n$ goes to infinity.

\end{corollary}

\begin{proof} The projector $p_n$ has rank at most $\lceil q\rceil+N$ so, after extraction, we may assume that its rank equals a constant $d\,.$ Then for each $n$ there is an orthonormal family $(\varphi^1_n,\cdots,\varphi^d_n)$ of eigenfunctions of $ D_{V,\gamma_n}$ with associated eigenvalues $\lambda^1_n,\cdots,\lambda^d_n\in [0,1-\frac{e_1}{2}]$ and such that $p_n=\sum_{i=1}^d \vert \varphi^i_n><\varphi^i_n\vert\,.$ There is also a hermitian matrix $G_n=(G^{ij}_n)_{1\leq i,j\leq d}$ with $0\leq G_n\leq {\bf 1}_d\,$ and
$p_n \gamma_n p_n=\sum_{1\leq i,j\leq d}G_n^{ij} \,\vert \varphi^i_n><\varphi^j_n\vert\;.$\medskip

After extraction, we may assume that for each $i,j$ the sequence of coefficients $(G^{ij}_n)_{n\geq 0}$ has a limit $G^{ij}_*\,.$ Moreover, arguing as in [Esteban-S. '99, Proof of Lemma 2.1 (b) p. 514-515], one shows that, after extraction, for each $i$ the sequence $(\varphi^i_n)_{n\geq 0}$ has a limit $\varphi^i_*$ for the strong topology of $H^{1/2}({\R}^3,{\bf C}^4)\,.$ The corollary is proved, taking
$\gamma_*:=\sum_{1\leq i,j\leq d} G_*^{ij} \,\vert \varphi^i_*><\varphi^j_*\vert\;.$

\end{proof}

We are now ready to prove Proposition  \ref{perturb}:

\begin{proof}
As a consequence of Corollary \ref{limit}, ${\cal E}_{DF}(\gamma_n)$ converges to ${\cal E}_{DF}(\gamma_*)$ and from Lemma \ref{regularity} (continuity of $Q$), $\,P^+_{\gamma_n}-P^+_{\gamma_*}$ converges to zero
for the norm of ${\cal B}({\cal H},{\cal F})\;.$ So $P^+_{\gamma_*}\gamma_*=\gamma_*$ and $\gamma_*$ is a minimizer of ${\cal E}_{DF}-\tr_{\mathcal H}$ on $\Gamma^+_{\leq q}\;.$ For any such minimizer, applying Lemma \ref{minimizing} we get
$${\rm tr}\big(( D_{V,\gamma_*}-1)\gamma_*\big)\,=\inf_{\gamma\in\Gamma_{\leq q}\,,\;\gamma =  P^+_{\gamma_*}\gamma}{\rm tr}\big(( D_{V,\gamma_*}-1)\gamma \big)\,.$$
This immediately implies that $\gamma_*=\1_{(0,\mu)}( D_{V,\gamma_*})+\delta\,$ with
$\, 0\leq \delta \leq \1_{\{\mu\}}( D_{V,\gamma_*})$
where $\mu=\lambda_{\lceil q\rceil,\gamma_*}$ is the $\lceil q\rceil$-th positive eigenvalue of $ D_{V,\gamma_*}$. Moreover $\tr_{\mathcal H}(\gamma_*)=q$ since $\mu\leq 1-e<1\;.$ Now, let $\psi$ be a normalized eigenvector of $ D_{V,\gamma_*}$ with eigenvalue $\lambda\in(1-e,1)$. Then $\gamma_*\psi=0$ and for $h\in(0,1)$ the density operator $\gamma(h):=\gamma_*+h\vert\psi><\psi\vert$ belongs to $\Gamma_{q+h}$ and satisfies $\gamma(h)=P^+_{V,\gamma_*}\gamma(h)P^+_{V,\gamma_*}$. So, taking $r>0$ such that the assumptions of Corollary \ref{nonempty} are satisfied and choosing $R,\,\rho$ as in this corollary, we find from \eqref{blocks} that for $h$ positive and small,
\[E_{q+h}\leq(\mathcal E_{DF}-\tr_{\mathcal H})\circ\theta(\gamma(h))=E_q+(\lambda-1) h+o(h)<E_q\,.\]
This ends the proof of Proposition \ref{perturb}.
\end{proof}

It remains to study the ground state problem for neutral molecules. We already proved the strict binding inequalities \eqref{cc} for $q=Z$ (see Corollary \ref{atoms}). So the case $q=Z$ of Theorem \ref{main} is a direct consequence of the following more general statement:

\begin{prop}\label{binding and ground state}

Assume that $Z$, $q$, $1-\kappa-{\pi\over 4}\alpha\,q\,$ are positive and that conditions \eqref{condition} and \eqref{cc} are satisfied.
Then there exists an admissible Dirac-Fock density operator $\gamma_*\in \Gamma_{q}^+$ such that
$$ {\cal E}_{DF}(\gamma_*)-\tr_{\mathcal H}(\gamma_*)=E_q \,.$$

For any such minimizer, there is an energy level $\mu\in (0,1]$ such that
\begin{equation}\label{Euler-Lagrangebis}
\gamma_*=\1_{(0,\mu)}( D_{V,\gamma_*})+\delta\,\ \hbox{ with }
\ 0\leq \delta \leq \1_{\{\mu\}}( D_{V,\gamma_*})\,.
\end{equation}

\end{prop}

When $q<Z$, Proposition \ref{binding and ground state} does not give any new information compared to Proposition \ref{perturb}. So we just need to prove Proposition \ref{binding and ground state} in the case $q\geq Z$. In order to do this, we perturb the nuclear charge distribution. We first introduce a function $G\in C^\infty_c(\R_+)$ with $G(r)\geq 0$ for all $r\geq 0$, $G(r)=0$ when $0\leq r \leq 1$ or $r \geq 4$, $G(r)=1$ for $2\leq r \leq 3$ and $4\pi\int_0^\infty G(r)r^2dr=1$. Then, to any positive integer $\ell$ we associate the function $g_\ell(x):=\ell^{-3} G(\vert x\vert/\ell)$ and the perturbed charge distribution
$\mathfrak n_\ell:=\mathfrak n+(q-Z+\ell^{-1}) g_\ell$.  The measure $\mathfrak n_\ell$ is positive and one has
$Z_\ell:=\mathfrak n_\ell(\R^3)=q+\ell^{-1}>q$. The corresponding perturbed Coulomb potential is $V_\ell=-\alpha\, \mathfrak n_\ell*\frac{1}{\vert \cdot\vert}$. Note that $V_\ell-V$ is radial and satisfies $-\frac{q-Z+\ell^{-1}}{\vert x\vert}\leq (V_\ell-V)(x) \leq 0$ for $\vert x\vert\geq \ell$ and $-\frac{q-Z+\ell^{-1}}{\ell}\leq (V_\ell-V)(x) \leq 0$ for $\vert x\vert\leq \ell$, so $\Vert V_\ell-V\Vert_\infty\leq \frac{q-Z+\ell^{-1}}{\ell}$, hence $\lim_{\ell\to\infty} \Vert V_\ell-V\Vert_\infty=0$.\medskip

From what we have just seen, if the constants $Z$, $\lambda_0:=1-\alpha \max(Z,q)$ and $\kappa:=\Vert V  D^{-1}\Vert_{{\cal B}({\cal H})}+2\alpha\,q\,$ satisfy \eqref{condition} with $q\geq Z$, then for $\ell$ large enough, the modified constants $Z_\ell$, $\lambda'_\ell:=1-\alpha \max(Z_\ell,q)$ and $\kappa'_\ell:= \Vert V_\ell  D^{-1}\Vert_{{\cal B}({\cal H})}+2\alpha\,q\,$ also satisfy \eqref{condition} with, in addition, $q<Z_\ell$.
So we may apply Proposition \ref{perturb} to the Dirac-Fock problem with nuclear charge density $\mathfrak n_\ell$. This gives us the existence of a Dirac-Fock ground state $\gamma_*^{\ell}$ of particle number $q$ in the external field $V_\ell$.\medskip

We now study the behavior of the minimizers $\gamma_*^{\ell}\,$ when $\ell\to+\infty$. First of all, since $\Vert V_\ell-V\Vert_\infty\to 0$, ${\cal E}^{\ell}_{DF} \to {\cal E}_{DF}$ uniformly on $\Gamma_{\leq q}\,,$ so the DF ground state energy associated to the potential $V_\ell$ converges to the DF ground state energy $E_q$ associated to $V$. In other words,
$$\lim\limits_{\ell\to+\infty} \left(\mathcal E^{\ell}_{DF}(\gamma_*^{\ell})-\tr_{\mathcal H}(\gamma_*^{\ell})\,\right)=E_q\,.$$

Moreover we have the following local compactness result:

\begin{lemma}\label{infty}

Under the above assumptions and notations, after extraction of a subsequence still denoted by $(\gamma_*^{\ell})\,$, there exist a density operator $\gamma_*\in \Gamma_{\leq q}$ and a sequence of positive numbers $R_\ell$ with $\lim R_\ell=+\infty$, such that for any smooth, compactly supported function $\eta\in C^\infty_c({\R}^3,{\R})\,,$ the integral operator with kernel $\eta(R_\ell^{-1}x)\,\bigl(\gamma_*^{\ell}-\gamma_*\bigr)(x,y)\,\eta(R_\ell^{-1}y)$ converges to zero for the topology of $X$ as $\ell$ goes to infinity.

\end{lemma}

\begin{proof} Since $0\leq\gamma^{\ell}_*\leq \1_{(0,1)}( D_{V_\ell,\gamma^{\ell}_*})$, Lemma \ref{estim-proj} implies that the operator $S^{\ell}:=(1-\Delta )^{\frac{1}{2}}\gamma^{\ell}_* (1-\Delta )^{\frac{1}{2}}$ is bounded in $\sigma_1(\mathcal H)$ independently of $\ell$, so after extraction it has a weak-$^*$ limit $S=(1-\Delta )^{1/2}\gamma_* (1-\Delta )^{1/2}$ as $\ell\to\infty$. Consider a function $\eta_0\in C^\infty_c({\R}^3,{\R})$ such that $\eta_0\equiv 1$ on $B(0,1)$. For any $\rho>0$, the operator $K_\rho=(1-\Delta )^{1/4}\eta_0(\rho^{-1}\cdot)(1-\Delta )^{-1/2}$ is compact. This implies that $\lim_{\ell\to\infty}\Vert K_\rho (S^\ell-S) K_\rho^*\Vert_{\sigma_1(\mathcal H)}=0\,$  (see {\it e.g.} \cite[Lemma 9]{Lewin} for a similar argument). Then, we may choose a sequence of positive numbers $\rho_\ell$ such that $\lim_{\ell\to\infty}\rho_\ell=+\infty$ and $\lim_{\ell\to\infty}\Vert K_{\rho_\ell} (S^\ell-S) K_{\rho_\ell}^*\Vert_{\sigma_1(\mathcal H)}=0\,$: for this, we just need the growth of $\rho_\ell$ to be sufficiently slow. Now, define $R_\ell:=\rho_\ell^{1/2}$. For any $\eta\in C^\infty_c({\R}^3,{\R})\,$, there is $\ell_0$ such that for all $\ell\geq \ell_0$ and $x\in\R^3$, $\eta(R_\ell^{-1}x)\eta_0(\rho_\ell^{-1}x)=\eta(R_\ell^{-1}x)$. Moreover the operator $L_{R_\ell}:=(1-\Delta )^{1/4}\eta(R_\ell^{-1}\cdot)(1-\Delta )^{-1/4}$ is bounded independently of $\ell$. So $\lim_{\ell\to\infty}\Vert L_{R_\ell} K_{\rho_\ell} (S^\ell-S) K_{\rho_\ell}^*L_{R_\ell}^*\Vert_{\sigma_1(\mathcal H)}=0\,.$ But for $\ell\geq \ell_0$ one has
$L_{R_\ell} K_{\rho_\ell} (S^\ell-S) K_{\rho_\ell}^*L_{R_\ell}^*=(1-\Delta )^{1/4}\eta(R_\ell^{-1}\cdot)(\gamma^{\ell}_*-\gamma_*)\eta(R_\ell^{-1}\cdot)(1-\Delta )^{1/4}$, so the lemma is proved.
\end{proof}

We now introduce two radial cut-off functions $\chi_{\epsilon}\in C^\infty(\R^3,\R_+)$ ($\epsilon=0,1$) such that $\,\chi_{_0}(x)=0$ for $\vert x\vert \geq 2\,$, $\chi_{_1}(x)=0$ for $\vert x\vert \leq 1$ and
$\chi_0^2+\chi_1^2=1\,$. We define the dilated cut-off functions $\chi_{\epsilon,\ell}(x)=\chi_{\epsilon}(R_\ell^{-1}x)$ and the associated localized density operators
\[\gamma_{\epsilon}^{\ell}(x,y):=\chi_{\epsilon,\ell}(x)\gamma_*^{\ell}(x,y)\chi_{\epsilon,\ell}(y)\,,\ \ \epsilon\in\{0,1\}\,.
\]
We have the following result:

\begin{lemma}\label{weak-continuity}

Assume that $\gamma_*^{\ell}\in X$ converges to $\gamma_*$ in the local sense of Lemma \ref{infty} as $\ell\to \infty$.
Then $\gamma_0^{\ell}$, $\gamma_1^{\ell}$ belong to $\Gamma_{\leq q}$ and
one has
\begin{equation}\label{splitting}
\tr_{\mathcal H} \gamma_*^{\ell}=\tr_{\mathcal H} \gamma_0^{\ell}+\tr_{\mathcal H} \gamma_1^{\ell}\,,\ \ 
\lim\{\mathcal E_{DF}^{\ell}(\gamma_*^{\ell})-\mathcal E_{DF}^{\ell}(\gamma_0^{\ell})-\mathcal E_{DF}^{\ell}(\gamma_1^{\ell})\}=0\,,
\end{equation}
\begin{equation}\label{com}
\lim_{\ell\to\infty}\Big\Vert  D_{V_\ell, \gamma_\epsilon^\ell} \chi_{\epsilon, \ell}  -\chi_{\epsilon, \ell}  D_{V_\ell, \gamma_*^\ell}\Big\Vert_{\mathcal B(\mathcal H)}=0\,,\quad \epsilon=0,1\,.
\end{equation}
\end{lemma}

\begin{proof}

The statement \eqref{splitting} is in the spirit of the concentration-compactness theory of P.L. Lions \cite{Lions-84} (dichotomy case). Its proof presents some similarities with the proof of Lemma 4 in \cite{HLS3} but it is less technical, as the present functional framework is simpler.\medskip

Obviously, one has
$$
\tr_{\mathcal H}\left(\gamma_*^{\ell}\right)=\tr_{\mathcal H}\left(\gamma_*^{\ell} (\chi_{0, \ell})^2\right)+\tr_{\mathcal H}\left(\gamma_*^{\ell} (\chi_{1, \ell})^2\right)=\tr_{\mathcal H}\left(\gamma_0^{\ell}\right)+\tr_{\mathcal H}\left(\gamma_1^{\ell}\right) \text {. }
$$

Let $\zeta(x):=\chi_0\left(\frac{2}{5} x\right) \chi_1(4 x)$. Then $\zeta \in C_c^{\infty}\left(\mathbb{R}^3, \mathbb{R}\right)$, $0 \leq \zeta \leq 1$, $\zeta(x)=1$ for $\frac{1}{2} \leq|x| \leq \frac{5}{2}$ and $\zeta(x)=0$ for $|x| \leqslant \frac{1}{4}$ or $|x| \geqslant 5$.
We introduce the dilated function $\zeta_\ell(x):=\zeta(R_\ell^{-1}x)$ and the associated integral operator
$$\gamma_2^\ell(x, y):=\zeta_\ell(x)\, \gamma_*^{\ell}(x,y)\, \zeta_\ell(y)\,.$$
From Lemma \ref{infty},
$$\lim _{\ell \rightarrow \infty}\left\|\gamma_2^{\ell}-\zeta_\ell(x) \gamma_*\left(x, y\right) \,\zeta_\ell(y)\right\|_X=0\,.$$
Moreover, recalling our notation $\mathcal F=H^{1 / 2}(\R^3,\C^4)$, we may write a decomposition of the form $\gamma_*=\sum_{n\geq 1} c_n\left|\psi_{n}\right\rangle\left\langle\psi_{n}\right|$ with
\[\langle\psi_n,\psi_{n'}\rangle_{\mathcal F}=\delta_{n,n'}\,, \ c_n\geq 0\ \hbox{ and } \ \sum_{n\geq 1} c_n=\left\|\gamma_*\right\|_X<\infty\,.\]
Then for each $n$, $\lim_{\ell\to\infty}\Vert \zeta_\ell \psi_n\Vert_{\mathcal F}=0$, since $\zeta$ vanishes on $B(0,1/4)$. In addition, there is $C>0$ such that, for all $\ell\geq 1$ and $\psi\in \mathcal F$,  $\Vert \zeta_\ell \psi\Vert_{\mathcal F}\leq C \Vert \psi\Vert_{\mathcal F}$. So, when $\ell\to\infty$, Lebesgue's dominated convergence theorem tells us that
$$\left\|\zeta_\ell(x) \gamma_*(x, y) \zeta_\ell(y)\right\|_X= \sum_{n\geq 1} c_n\left\|\zeta_\ell\psi_{n}\right\|_{\mathcal F}^2\to 0\,,$$
hence $\lim _{\ell \rightarrow \infty}\left\|\gamma_2^{\ell}\right\|_X=0$.
So, using inequality \eqref{hardy1}, we find that the norms $\left\|\rho_{\gamma_2^{\ell}} * \frac{1}{|\cdot |}\right\|_{L^{\infty}(\mathbb{R}^3)},$ $\left\|W_{\gamma_2^{\ell}}\right\|_{\mathcal B(\mathcal H)}$ and $\left\|\frac{\gamma_2^\ell(x, y)}{|x-y|}\right\|_{\mathcal B(\mathcal H)}$ converge to $0$ as $\ell\to\infty$.\medskip

Now, we write
$$
\mathcal E_{D F}^{\ell}\left(\gamma_*^\ell\right)-\mathcal E_{D F}^{\ell}\left(\gamma_0^\ell\right)-\mathcal E_{D F}^{\ell}\left(\gamma_1^\ell\right)=A_\ell+B_\ell
$$
where
$$A_\ell:=\frac{i}{R_\ell} \sum_{\epsilon=0}^1\operatorname{tr}_{\mathcal H}\left\{\left(\alpha \cdot \nabla \chi_\epsilon\right)(R_\ell^{-1} x) \gamma_*^\ell(x,y)\, \chi_{\epsilon, \ell}(y)\right\}=\mathcal O\left(\frac{1}{R_l}\right)$$
and
$$B_\ell:=\alpha \iint_{\mathbb{R}^3 \times\mathbb{R}^3} \frac{(\chi_{0, \ell})^2(x) (\chi_{1, \ell})^2(y)}{|x-y|}\left(\rho_{\gamma_*^\ell}(x) \rho_{\gamma_*^\ell}(y)-\left|\gamma_*^\ell(x, y)\right|^{2}\right) d^3 x\, d^3 y\,.$$
We have
$$
\begin{aligned}
\frac{\chi_{0, \ell}(x) \chi_{1, \ell}(y)}{|x-y|} & =\frac{\chi_{0, \ell}(x) \chi_{1, \ell}(y)}{|x-y|}\left(\1_{\big\{|x-y| \leq \frac{R_\ell}{2}\big\}} 
+\1_{ \big\{|x-y| > \frac{R_\ell}{2}\big\}} \right) \\
&\leq \frac{1}{|x-y|}\1_{\big\{ \frac{R_\ell}{2}\leq |x| \leq \frac{5 R_\ell}{2}\big\}}\1_{\big\{ \frac{R_\ell}{2}\leq |y| \leq \frac{5 R_\ell}{2}\big\}}+\frac{2}{R_\ell}\,,
\end{aligned}
$$
hence
\begin{equation}\label{xy}
\frac{\chi_{0, \ell}(x) \chi_{1, \ell}(y)}{|x-y|}\leq \frac{(\zeta_\ell)^2(x) (\zeta_\ell)^2(y)}{|x-y|} +\frac{2}{R_\ell}\,.
\end{equation}
In addition, we have the inequalities $\,0\leq \rho_{\gamma_*^\ell}(x) \rho_{\gamma_*^\ell}(y)-\left|\gamma_*^\ell(x, y)\right|^{2} \leq \rho_{\gamma_*^\ell}(x) \rho_{\gamma_*^\ell}(y)$ and the identity $(\zeta_\ell)^2(x) (\zeta_\ell)^2(y)\rho_{\gamma_*^\ell}(x) \rho_{\gamma_*^\ell}(y)=\rho_{\gamma_2^\ell}(x) \rho_{\gamma_2^\ell}(y)$.
As a consequence, we get the estimate
$$
0 \leqslant B_\ell \leqslant \alpha \iint_{\mathbb{R}^3 \times \mathbb{R}^3} \frac{\rho_{\gamma_2^\ell}(x) \rho_{\gamma_2^\ell}(y)}{|x-y|}+\mathcal O\left(\frac{1}{R_\ell}\right)\leq\, \alpha q\,\Big\|\rho_{\gamma_2^{\ell}} * \frac{1}{|\cdot |}\Big\|_{L^{\infty}(\mathbb{R}^3)} + \mathcal O\left(\frac{1}{R_\ell}\right)\,,$$
so \eqref{splitting} is proved.\medskip

In order to prove \eqref{com} one writes
\begin{equation}\label{commut}
 D_{V_\ell, \gamma_0^\ell} \chi_{0, \ell}  -\chi_{0, \ell}  D_{V_\ell, \gamma_*^\ell}  =\left[ D_{V_\ell, \gamma_*^\ell}\,,\,\chi_{0, \ell}\right]  -\alpha W_{\gamma_*^\ell-\gamma_0^\ell} \,\chi_{0, \ell}\,.
\end{equation}
One has
$$
\left[ D_{V_\ell, \gamma_*^\ell}\,,\, \chi_{0, \ell}\right]=\frac{-i}{R_\ell}\left(\alpha \cdot \nabla \chi_{0}\right)(R_\ell^{-1} x )+\alpha \frac{\chi_{0,\ell}(y)-\chi_{0,\ell}(x)}{|x-y|} \gamma_*^\ell(x, y) \,,
$$
so
\begin{equation}\label{commut1}
\left\|\left[ D_{V_\ell, \gamma_*^\ell}\,,\, \chi_{0, \ell}\right]\right\|_{\mathcal B(\mathcal H)}=\mathcal O\left(\frac{1}{R_\ell}\right) \,.
\end{equation}
Now, for any test function $\psi \in C_c^{\infty}\left(\mathbb{R}^3, \mathbb{C}^4\right)$,
$$
\begin{aligned}
(W_{\gamma_*^\ell-\gamma_0^\ell} \chi_{0, \ell} \,\psi)(x) 
=&\int_{\R^3} \frac{\chi_{0, \ell}(x) ( \chi_{1, \ell})^2(y)  \rho_{\gamma_*^\ell}(y)\psi(x)}{|x-y|}  d^3 y \\
& -\int_{\R^3} \frac{\left(1-\chi_{0, \ell}(x) \chi_{0,\ell}(y)\right)\chi_{0, \ell}(y) \gamma_*^\ell(x, y) \psi(y)}{|x-y|} d^3 y\,.
\end{aligned}
$$
Using \eqref{xy} once again, one gets
$$\Big\Vert \int_{\R^3} \frac{\chi_{0, \ell}(x) ( \chi_{1, \ell})^2(y)  \rho_{\gamma_*^\ell}(y)\psi(x)}{|x-y|}  d^3 y\Big\Vert_{L^2(d^3 x)}\leq \left(\Big\|\rho_{\gamma_2^{\ell}} * \frac{1}{|\cdot |}\Big\|_{L^{\infty}(\mathbb{R}^3)}+\frac{2q}{R_\ell}\right)\Vert \psi \Vert_{\mathcal H}\,.$$
Moreover, arguing as in the proof of \eqref{xy}, one easily gets
\begin{equation}\label{xyz}
\frac{\left(1-\chi_{0, \ell}(x) \chi_{0,\ell}(y)\right)\chi_{0, \ell}(y)}{|x-y|}\leq \frac{\zeta_\ell(x) \zeta_\ell(y)}{|x-y|} +\frac{2}{R_\ell}\,,
\end{equation}
hence
$$
\begin{aligned}
\Big\Vert \int_{\R^3} \frac{\left(1-\chi_{0, \ell}(x) \chi_{0,\ell}(y)\right)\chi_{0, \ell}(y) \gamma_*^\ell(x, y) \psi(y)}{|x-y|}& d^3 y\Big\Vert_{L^2(d^3 x)}\\
&\leq \left(\Big\|\frac{\gamma_2^\ell(x, y)}{|x-y|}\Big\|_{\mathcal B(\mathcal H)}+\frac{2q}{R_\ell}\right)\Vert \psi \Vert_{\mathcal H}\,.
\end{aligned}
$$
The above estimates imply that $\lim_{\ell\to\infty} \Vert W_{\gamma_*^\ell-\gamma_0^\ell} \chi_{0, \ell}\Vert_{\mathcal B(\mathcal H)} = 0\,.$ Combining this with \eqref{commut} and \eqref{commut1} one gets \eqref{com} for $\epsilon=0$. The case $\epsilon=1$ is proved in the same way.

 \end{proof}
 
 Before proving Proposition \ref{binding and ground state} we need a last lemma:
 \begin{lemma}\label{projector-liminf}
 Assume that $\gamma_*^{\ell}\in X$ converges to $\gamma_*$ in the local sense of Lemma \ref{infty} as $\ell\to \infty$. Then:
 \begin{equation}\label{projector}
P^-_{V,\gamma_*}\gamma_*=0\,,
\end{equation}
\begin{equation}\label{liminf}
\liminf_{\ell\to\infty}\big(\mathcal E_{DF}^{\ell}-\tr_{\mathcal H}\big)(\gamma_1^{\ell})\geq 0\,.
\end{equation}

\end{lemma}

\begin{proof}

 Let $\xi(\ell):=\Vert V_\ell-V\Vert_{L^\infty(\R^3)}+\max\limits_{\epsilon=0,1}\Big\Vert  D_{V_\ell, \gamma_\epsilon^\ell} \chi_{\epsilon, \ell}  -\chi_{\epsilon, \ell}  D_{V_\ell, \gamma_*^\ell}\Big\Vert_{\mathcal B(\mathcal H)}$.
 
 \noindent From the definition of $V_\ell$ and from \eqref{com}, we know that $\lim _{\ell \rightarrow \infty} \xi(\ell)=0$. From the Euler-Lagrange equation satisfied by $\gamma_*^\ell\,$, there is a (finite or infinite) set $I_\ell$ of integers and an orthonormal sequence $(\psi^\ell_{n})_{n\in I_\ell}$ of common eigenvectors of $\gamma_*^\ell$ and $ D_{V_\ell, \gamma_*^\ell}\,$, satisfying:
$$
\begin{aligned}
&  D_{V_\ell, \gamma_*^\ell} \psi^\ell_{n}=\lambda_{n}^\ell \psi_{n}^\ell\,, \quad \gamma_*^\ell=\sum_{n\in I_\ell} g_n^\ell\left|\psi_{n}^\ell\right\rangle\left\langle\psi_{n}^\ell\right|,\quad \left\langle\psi_{n}^\ell, \psi_{n'}^\ell\right\rangle_{\mathcal H}=\delta_{n, n'}\,, \\
& 0 < \lambda_n^\ell < 1\,,\quad 0 < g_n^\ell\leq 1\,, \quad \sum_{n\in I_\ell} g_n^\ell=\tr_{\mathcal H}\left(\gamma_*^\ell\right)=q \,.
\end{aligned}
$$
Then $\gamma_{\epsilon}^\ell=\sum_{n\in I_\ell} g_n^\ell\left|\psi_{\epsilon, n}^\ell\right\rangle\left\langle\psi_{\epsilon, n}^\ell\right|$ with $\psi_{\epsilon, n}^\ell(x)=\chi_{\epsilon, \ell}(x) \psi_{n}^\ell(x)$, $\epsilon=0,1$. Moreover,
$$\tr_{\mathcal H}\left(\gamma_{\epsilon}^\ell\right)=\left\|\gamma_{\epsilon}^\ell\right\|_{\sigma_1 (\mathcal H)}=\sum_{n\in I_\ell} g_n^\ell\left\|\psi_{\epsilon,n}^\ell\right\|_{\mathcal H}^2,
\quad
\left\|\gamma_{\epsilon}^\ell\right\|_X=\sum_{n\in I_\ell} g_n^\ell\left\|\psi_{\epsilon, n}^\ell\right\|_{\mathcal F}^2 \text {. }
$$
For $n\in I_\ell$ we have $\left\|\left(D_{V, \gamma_0^\ell }-\lambda_n^\ell\right) \psi_{0, n}^\ell\right\|_{\mathcal H} \leq \xi(\ell)$. On the other hand, \eqref{hardy4} implies that
$\Vert \,(\,\vert  D_{V,\gamma_0^\ell}\vert+\lambda_n^\ell)^{-1} \Vert_{\mathcal B(\mathcal H)}\leq \frac{1}{\lambda_0}$
with $\lambda_0=1-\alpha \max(q,Z)>0$. Thus,
 $$\left\|P_{V, \gamma_0^\ell}^{-} \psi_{0, n}^\ell\right\|_{\mathcal H} \leq \left\Vert \,\big(\,\vert  D_{V,\gamma_0^\ell}\vert+\lambda_n^\ell\big)^{-1} \right\Vert_{\mathcal B(\mathcal H)}\left\|P_{V, \gamma_0^\ell}^{-}\left(D_{V, \gamma_0^\ell }-\lambda_n^\ell\right) \psi_{0, n}^\ell\right\|_{\mathcal H}\leq\frac{\xi(\ell)}{\lambda_0} \,.$$
As a consequence,
$$
\left\|P_{V, \gamma_0^\ell}^{-} \gamma_{0}^\ell\right\|_{\sigma_1(\mathcal H)} \leqslant \sum_{n\in I_\ell}g_n^\ell\left\|P_{V, \gamma_0^\ell}^{-} \psi_{0, n}^\ell\right\|_{\mathcal H}\,\| \psi_{0, n}^\ell \|_{\mathcal H} \,\leqslant q\, \frac{\xi(\ell)}{\lambda_0}=o(1)_{\ell\to\infty}\,.\\
$$
Then, recalling that $\lim_{\ell\to \infty}\Vert \gamma_0^\ell-\gamma_*\Vert_X=0$ and using \eqref{Q lipschitz}, we get \eqref{projector}.\medskip

In order to prove \eqref{liminf}, we write
$$\begin{aligned}
\tr\Big(  D_{V, \gamma_1^\ell} \gamma_1^\ell\,(\Lambda^{+}-\Lambda^{-})\Big)
= \,&\tr\Big(  D_{0, \gamma_1^\ell} \Lambda^{+} \gamma_1^\ell\,\Lambda^{+}\Big)-\tr\Big(  D_{0, \gamma_1^\ell} \Lambda^{-} \gamma_1^\ell\,\Lambda^{-}\Big)\\
&+\tr\Big((V\,\chi_{1,\ell})\gamma_*^\ell\,\chi_{1,\ell}\,(\Lambda^{+}-\Lambda^{-})\Big)    \,.
\end{aligned}
$$
We have
$$
\tr\Big(  D_{0, \gamma_1^\ell} \Lambda^{+} \gamma_1^\ell\,\Lambda^{+}\Big)\geqslant  \tr\Big(  D \Lambda^{+} \gamma_1^\ell\,\Lambda^{+}\Big)=\|\Lambda^{+} \gamma_1^\ell \Lambda^{+}\|_X\,.
$$
Moreover, using Tix' inequality \cite{Tix-98} one gets
$$-\tr\Big(  D_{0, \gamma_1^\ell} \Lambda^{-} \gamma_1^\ell\,\Lambda^{-}\Big)\geq \left(1-\alpha \big(\frac{\pi}{4}+\frac{1}{\pi}\big) q\right)\big\|\Lambda^{-} \gamma_1^\ell \Lambda^{-}\big\|_X\,.$$
In addition, one has
$$
\begin{aligned}
\Vert V\,\chi_{1,\ell} D^{-1}\Vert_{\mathcal B(\mathcal H)}&=\Big\Vert \big((\mathfrak n \1_{\vert \cdot\vert \leq R_\ell/2})*\vert\cdot\vert^{-1} + (\mathfrak n \1_{\vert \cdot\vert > R_\ell/2})*\vert\cdot\vert^{-1}\big)\chi_{1,\ell} D^{-1}\Big\Vert_{\mathcal B(\mathcal H)}\\
&\leq \frac{2 Z}{R_\ell}+2\mathfrak n\big(\R^3\setminus B(0,R_\ell/2)\,\big)=o(1)_{\ell\to\infty}
\end{aligned}$$
and the Euler-Lagrange equation satisfied by $\gamma_*^\ell$ implies that $\Vert D \gamma_*^\ell D\Vert_{\sigma_1(\mathcal H)}=\mathcal O(1)\,$, so
$$\lim_{\ell\to\infty} \tr\Big((V\,\chi_{1,\ell})\gamma_*^\ell\,\chi_{1,\ell}\,(\Lambda^{+}-\Lambda^{-})\Big)=0\,.$$
Gathering these informations, we get the lower estimate
$$\tr\Big(  D_{V, \gamma_1^\ell} \gamma_1^\ell\,(\Lambda^{+}-\Lambda^{-})\Big)\geq \big\|\Lambda^{+} \gamma_1^\ell \Lambda^{+}\big\|_X\,+
\Big(1-\alpha \big(\frac{\pi}{4}+\frac{1}{\pi}\big) q\Big)\big\|\Lambda^{-} \gamma_1^\ell \Lambda^{-}\big\|_X+o(1)_{\ell\to\infty}\,.$$

On the other hand, we may write $\,\tr\Big(  D_{V, \gamma_1^\ell} \gamma_1^\ell\,(\Lambda^{+}-\Lambda^{-})\Big)=I_\ell+J_\ell+K_\ell$ with
$$
\begin{aligned}
I_\ell:=&\tr_{\mathcal H}\left(\big( D_{V, \gamma_1^\ell} \chi_{1, \ell}  -\chi_{1, \ell}  D_{V_\ell, \gamma_*^\ell}\big) \,\gamma_*^\ell \chi_{1,\ell}\big(\Lambda^{+}-\Lambda^{-}\big)\right) \\
J_\ell:=& \tr_{\mathcal H}\big(\Lambda^{+}\chi_{1, \ell}  D_{V_\ell, \gamma_*^\ell} \gamma_*^\ell \chi_{1, \ell}\Lambda^{+}\big) \,,\\
K_\ell:=& -\tr_{\mathcal H}\big(\Lambda^{-}\chi_{1, \ell}  D_{V_\ell, \gamma_*^\ell} \gamma_*^\ell \chi_{1, \ell}\Lambda^{-}\big) \,.
\end{aligned}
$$
From \eqref{com}, we have $\vert I_\ell \vert \leq q\,\xi(\ell)= o(1)_{\ell\to\infty}$. Moreover the Euler-Lagrange equation satisfied by $\gamma_*^\ell$ implies that $ D_{V_\ell, \gamma_*^\ell} \gamma_*^\ell$ is a self-adjoint operator satisfying $0\leq  D_{V_\ell, \gamma_*^\ell} \gamma_*^\ell\leq \gamma_*^\ell\,$. As a consequence,
$J_\ell\leq \tr_{\mathcal H} \big( \Lambda^{+}\gamma_1^\ell\Lambda^{+}\big)$ and $K_\ell \leq 0\,$, so
$$\tr\Big(  D_{V, \gamma_1^\ell} \gamma_1^\ell\,(\Lambda^{+}-\Lambda^{-})\Big)\leq \tr_{\mathcal H} \big( \Lambda^{+}\gamma_1^\ell\Lambda^{+}\big)+o(1)_{\ell\to\infty}\,.
$$

Combining our lower and upper estimates on $\tr\Big(  D_{V, \gamma_1^\ell} \gamma_1^\ell\,(\Lambda^{+}-\Lambda^{-})\Big)$ we conclude that

$$\Big(1-\alpha\Big(\frac{\pi}{4}+\frac{1}{\pi}\Big) q\Big)\left\|\Lambda^- \gamma_1^\ell \Lambda^-\right\|_{X}+\|\Lambda^{+} \gamma_1^\ell \Lambda^{+}\|_X - \tr_{\mathcal H} \big( \Lambda^{+}\gamma_1^\ell\Lambda^{+}\big)\leqslant o(1)_{\ell\to\infty}\,.
$$
But $\left(1-\alpha\left(\frac{\pi}{4}+\frac{1}{\pi}\right) q\right)\left\|\Lambda^- \gamma_1^\ell \Lambda^-\right\|_{X}$ and $\Big(\|\Lambda^{+} \gamma_1^\ell \Lambda^{+}\|_X - \tr_{\mathcal H} \big( \Lambda^{+}\gamma_1^\ell\Lambda^{+}\big)\Big)$ are both nonnegative, so
$$\lim_{\ell\to\infty} \left\|\Lambda^- \gamma_1^\ell \Lambda^-\right\|_{X} = \lim_{\ell\to\infty} \Big(\|\Lambda^{+} \gamma_1^\ell \Lambda^{+}\|_X - \tr_{\mathcal H} \big( \Lambda^{+}\gamma_1^\ell\Lambda^{+}\big)\Big)=0\,.$$
As a consequence,
$$
\begin{aligned}
\big(\mathcal E_{D F}^{\ell}-\tr_{\mathcal H}\big)\big(\gamma_1^\ell\big) &\geqslant \tr\big( D_{V_\ell} \gamma_1^\ell\big)-\tr_{\mathcal H}(\gamma_1^\ell)\\
&=\|\Lambda^{+} \gamma_1^\ell \Lambda^{+}\|_X -\tr_{\mathcal H} \big( \Lambda^{+}\gamma_1^\ell\Lambda^{+}\big) \\
&\quad- \left\|\Lambda^- \gamma_1^\ell \Lambda^-\right\|_{X}-\tr_{\mathcal H} \big( \Lambda^{-}\gamma_1^\ell\Lambda^{-}\big)+\tr\big((V_\ell\chi_{1,\ell})\gamma_*^\ell\chi_{1,\ell}\big)\\
&= o(1)_{\ell\to\infty}
\end{aligned}
$$
and \eqref{liminf} is proved.
 \end{proof}
 
Thanks to lemmas \ref{infty}, \ref{weak-continuity} and \ref{projector-liminf}, we are now ready to prove Proposition \ref{binding and ground state} for $q\geq Z$:

\begin{proof}
Recalling that $\lim_{\ell\to\infty} \left(\mathcal E^{\ell}_{DF}(\gamma_*^{\ell})-\tr_{\mathcal H} (\gamma_*^{\ell})\,\right)=E_q$, we deduce from \eqref{splitting} and \eqref{liminf} the inequality $\limsup_{\ell\to\infty} \left(\mathcal E^{\ell}_{DF}(\gamma_0^\ell)-\tr_{\mathcal H} (\gamma_0^\ell)\,\right)\leq E_q\,$.
But from Lemma \ref{infty}, we find that $\lim\Vert \gamma_0^{\ell}-\gamma_*\Vert_X=0$, so
\[{\mathcal E}_{DF}(\gamma_*)-\tr_{\mathcal H} (\gamma_*)=\lim_{\ell\to\infty} \left(\mathcal E_{DF}^{\ell}(\gamma_0^{\ell})-\tr_{\mathcal H} (\gamma_0^\ell)\,\right)\leq E_q\,.\]

On the other hand, with $q':=\tr_{\mathcal H}(\gamma_*)$ we have $q' =\lim_{\ell\to\infty} \tr_{\mathcal H} (\gamma_0^{\ell})\leq q$, and \eqref{projector} tells us that $\gamma_*$ is in $\Gamma_{\leq q'}^+\,$, hence ${\mathcal E}_{DF}(\gamma_*)-\tr_{\mathcal H} (\gamma_*)\geq E_{q'}\geq E_q\,$.\medskip

As a consequence, $\gamma_*$ is a minimizer of ${\mathcal E}_{DF}-\tr_{\mathcal H}$ both on $\Gamma^+_{\leq q'}$ and $\Gamma^+_{\leq q}$. Then the strict binding inequality \eqref{cc} implies that $q'=q$. Finally, applying Lemma \ref{minimizing} to the constant sequence $\gamma_n=\gamma_*$ we find that

$${\rm tr}\big(( D_{V,\gamma_*}-1)\gamma_*\big)=\min_{g\in\Gamma_{\leq q}\,,\,P^+_{\gamma_*}g=g}{\rm tr}\big(( D_{V,\gamma_*}-1)g\big)\,.$$

So $\gamma_*$ is of the form $p+\delta$ with $p=\1_{(0,\mu)}( D_{V,\gamma_*})\; {\rm and}\quad 0\leq \delta \leq \1_{\{\mu\}}( D_{V,\gamma_*})$ for some $0<\mu\leq 1\,.$\medskip

Proposition \ref{binding and ground state} is thus true. This ends the proof of Theorem \ref{main}.
\end{proof}

{\bf Acknowledgement.}

The author wishes to thank Isabelle Catto and Long Meng for carefully reading a preliminary version of this manuscript and making very useful remarks. In particular, Long Meng suggested an improvement of estimate \eqref{main-estimate} that increased the domain of validity of the method for neutral atoms (from $Z\leq 18$ to $Z\leq 22$) as well as for positive ions. The author is also grateful to the referees for their remarks that greatly improved the quality of this paper. This project has received funding from the Agence Nationale de la Recherche (grant agreement molQED).\bigskip

{\bf Declarations.}\bigskip

{\bf Conflict of interest} The author declares no conflict of interest.\bigskip

{\bf Data availability statement} Data sharing is not applicable to this article as it has no associated data.


\bibliographystyle{amsplain}

\end{document}